\documentclass[12pt,psamsfonts]{amsart}
\usepackage{amsmath}
\usepackage{amsthm}
\usepackage{amssymb}
\usepackage{amscd}
\usepackage{amsfonts}
\usepackage{amsbsy}
\usepackage{graphicx}
\usepackage[dvips]{psfrag}

\newtheorem {theorem*}{Theorem}
\newtheorem {theorem} {Theorem}
\newtheorem {proposition} [theorem]{Proposition}
\newtheorem {corollary} [theorem]{Corollary}
\newtheorem {lemma}  [theorem]{Lemma}
\newtheorem {example} [theorem]{Example}

\begin{document}

\title[ On structural stability of quasihomogeneous
vector fields] {\bf On the structural stability of planar quasihomogeneous polynomial vector fields}

\author[Regilene D. S. Oliveira  and Yulin Zhao]
{Regilene D. S. Oliveira  and Yulin Zhao}
\address{Regilene D. S. Oliveira,
Departamento de Matem\'atica, Instituto de Ci\^encias Matem\'aticas e
de Computa\c c\~ao,  Universidade de S\~ao Paulo,  Caixa Postal 668, 13560-970, Sao Carlos, SP, Brazil}

\email{regilene@icmc.usp.br}

\address{Yulin Zhao, Department of Mathematics, Sun Yat-sen University, Guangzhou, 510275, People's Republic of  China. }

\email{mcszyl@mail.sysu.edu.cn}

\thanks{ 
The first author is partially supported by FAPESP project: 2009/08774-0. The second author is  supported by the NSF of China (No.11171355), the Ph.D. Programs Foundation of Ministry of Education of China (No. 20100171110040)  and Program for New Century Excellent Talents in University. }

\date{}
\dedicatory{}

\maketitle

\maketitle

\begin{abstract}

Denote by $H_{pqm}$ the space of all planar $(p,q)$-quasihomogeneous vector fields of degree $m$ endowed with the coefficient topology. In this paper we characterize the set $\Omega_{pqm}$ of the vector fields in $H_{pqm}$ that are structurally stable with respect to perturbations in $H_{pqm}$, and  determine the exact number of the topological equivalence classes in $\Omega_{pqm}$. The  characterisation is applied to give an extension of the Hartman-Grobmann Theorem for such family of planar polynomial vector fields.  It follows from the main result in this paper that, for a given $X \in H_{pqm}$ we give a explicit method to decide whether it is structurally stable with respect to perturbation in $H_{pqm}$  before finding the vector field induced by $X$ in the Poincar\'e-Lyapunov sphere. This work is an extension and an improvement of the Llibre-Perez-Rodriguez's paper \cite{LRR}, where the homogeneous case was considered. More precisely, if both $p$ and $q$ are odd, the main results of this paper are similar to those of the Llibre-Perez-Rodriguez's paper; if either $p$ or $q$ is odd while the other is even, we present some results which do not appear in the above mentioned paper. For example, one of the  interesting results is that there may be  triples $(p,q,m)$ such that $H_{pqm}\not=\emptyset$ but $\Omega_{pqm}=\emptyset$, which does not occur in the homogeneous case.
\end{abstract}

\section{Introduction and statement of the main results}

The structural stability of planar vector fields has been a subject of great interest in the global qualitative theory of dynamical systems since the sixties. The first definition of structural stability of planar vector fields goes back to Andronov and Pontrjagin (1937) \cite{andronov} and Peixoto (1962) \cite{peixoto}. There are many paper about this subject. In 1990 \cite{shafer} Shafer characterized the planar gradient polynomial vector fields which are structurally stable with respect to perturbations in the set of all $C^r$ planar vector fields and in the set of all planar polynomial vector fields. In 1993, Jarque and Llibre \cite {JL} found the similar characterisation to planar Hamiltonian polynomial vector field with respect to perturbations in the same sets of Shafer. In 1996 Llibre, Perez and Rodriguez \cite{LRR} characterized the structural stable homogeneous vector fields with respect to perturbation in the same restricted set. In 2000 the same authors \cite{LRR1} extended  most of the results of the previous paper to systems of the form $X=(P_m, Q_n)$, where $P_m$ and $Q_n$ are polynomial functions of degree $m$ and $n$, respectively, $m, n > 1$. In 2005 Jarque, Llibre and Shafer \cite{JLS1} provided sufficient conditions for a planar polynomial foliation to be structurally stable under several different types of perturbation. In 2008 \cite{JLS2}, Jarque, Llibre and Shafer obtained the full characterisation of structural stability of polynomial foliations of degree 1 and 2, both in the Poincar\'e sphere and in the plane, together with a complete catalogue of phase portrait of stable systems.

 However the complete and explicit characterisation of all planar structurally stable polynomial vector fields of degree $m$ with respect to polynomial perturbations, is an open problem. In many papers about this subject we found the study of some families of polynomials vector fields, specially modulo limit cycles, as in \cite{AKL}.

A function $f(x,y)$ is called a {\it $(p,q)$-quasihomogeneous function of degree $m$} if $f(\lambda^p
x,\lambda^q y)=\lambda^m f(x,y)$ for all $\lambda\in \mathbb{R}$. If $P(x,y)$ and $Q(x,y)$ are $(p,q)$-
quasihomogeneous polynomials of degree $p-1+m$ and $q-1+m$, respectively, we say that $X=(P,Q)$ is a {\it planar $(p,q)$-quasihomogeneous polynomial
vector field of degree $m$}. The system of differential equations associated to $X$ is
\begin{equation}\label{eq1}
\frac{dx}{dt}=P(x,y)=\sum_{pi+qj=p-1+m}a_{ij}x^iy^j,\,\,\,\,\,\frac{dy}{dt}=Q(x,y)=\sum_{pi+qj=q-1+m}b_{ij}x^iy^j.
\end{equation}
Here $p, q$ and $m$ are positive integers and $P(x,y)$ and $Q(x,y)$ are coprime in the ring $\mathbb{R}[x,y]$. To be short we denote this by  $(P(x,y),Q(x,y))=1$.

Observe that the above definition is the natural one for the following reasons\cite{CGP}:
\begin{enumerate}
\item When $p=q=1$ it coincides with the usual definition of homogeneous vector
field of degree $m$.
\item The differential equation $dy/dx=Q/P$ associated with $X$, is invariant by the change of variables $\bar{x}=\lambda^p x, \bar{y}=\lambda^q y$.
\item Homogeneous vector fields can be integrated using polar coordinates where\-as $(p, q)$-quasihomogeneous vector fields can be integrated using the $(p, q)$-polar coordinates. These generalized polar coordinates were introduced by Lyapunov in his study of the stability of degenerate critical points\cite{DLA}. In this paper we will describe this change of coordinates and some of their main properties in Section \ref{s2}.
\end{enumerate}

Let $H_{pqm}$ be the set of all planar $(p,q)$-quasihomogeneous vector fields of degree $m$.

It follows from  \eqref{eq1} that  $H_{pqm}\not=\emptyset$ if and only if $pi+qj=p-1+m$ and $pi'+qj'=q-1+m$  have  non-negative integer solutions $(i,j)$ and $(i',j')$ respectively. Therefore if $(p,q,m)=(3,7,2)$, then $H_{372}=\emptyset$; and if $(p,q,m)=(1,2,2)$, then  $H_{122}\not=\emptyset$.

Although there are integers $p,q$ and $m$ such that $H_{pqm}=\emptyset$,  there has been a substantial amount of work devoted to understanding
the properties  of  the quasihomogeneous vector fields. For instance, the quasihomogeneous vector fields  have appeared in several works about the Hilbert's Sixteenth Problem\cite{CGP,LLYZ},  and about integrability of planar vector fields\cite{chavaga}. The $(p, q)$-polar coordinates, used to study the quasihomogeneous vector fields, have also been applied to study properties of planar differential equations\cite{BM,L}.  It is well known that (see \cite{chavaga}), given a polynomial $f \in \mathbb R[x, y]$ we can write  it in the form $f = f_m + f_{m+1} +...+f_{m+n}$, where $f_k$ is a $(p,q)$-quasihomogeneous polynomial function of degree $k$.

Before stating the main results of this paper,  we give some definition and notations.  To study the behaviour of the trajectories of a planar differential system near infinity we use the Poincar\'e-Lyapunov compactification, see for instance \cite{DLA}.  In the Poincar\'e-Lyapunov compactification we prefer to work on a hemisphere, calling it as Poincar\'e-Lyapunov disk.  The induced vector field in the Poincar\'e-Lyapunov disc is called {\it the Poincar\'e-Lyapunov compactification of the vector field $X$}, denoted by $E(X)$.

Roughly speaking, we shall say that two vector fields $X, Y \in H_{pqm}$ are {\it topologically equivalent} if there exists a homeomorphism $h$ in the Poincar\'e-Lyapunov disc, carrying orbits of the flow induced by $E(X)$ onto orbits of the flow induced by  $E(Y)$, preserving sense but not necessarily parametrization; $h$ is termed an equivalence homeomorphism between $X$ and $Y$. Moreover, as in the homogeneous case, the coefficients of $X$, $P$ and $Q$, are polynomials so, there exists a positive integer $k < \infty$, such that every $X =(P,Q) \in H_{pqm}$ could be identified with a unique point in $\mathbb R^{k}$, by the identification of the coefficients of $X$ with points from $\mathbb R^{k}$. The number $k$ could be chosen in the following way: let $k_1$ and $k_2$  be the  numbers of the non-negative integer solutions, in the case they exist, for the equations $pi+qj=p-1+m$ and $pi'+qj'=q-1+m$ respectively. Then $k=k_1+k_2$. We take in $H_{pqm}$ the topology induced by the Euclidean norm in $\mathbb R^{k}$.

Further,  a vector field  $X\in H_{pqm}$ is {\it structurally stable}  with respect to perturbation in $H_{pqm}$ if there exists a neighborhood $U$ of $X$ in $H_{pqm}$ such that for all $Y \in U$, $X$ and $Y$ are topologically equivalent. We also need to observe that this definition of structural stability does not require that the equivalence homeomorphism is near the identity map on the Poincar\'e-Lyapunov sphere. It is important to say that in certain cases,  for open manifolds, the restriction of $h$ to a neighborhood of identity in $H_{pqm}$ is not superfluous (see {\cite{dumortier-shafer, KKN, shafer}), but here it is redundant. This follows from the Peixoto results in \cite{peixoto}, as in the homogeneous case. Since the coefficient topology is equivalent to the $C^1$ topology and the Poincar\'e-Lyapunov sphere is a manifold in the Peixoto conditions we don't need to require that our equivalence homeomorphism be near the identity map.
Peixoto showed in\cite {peixoto} that on an orientable differentiable compact connected 2-manifold without boundary, if a $C^1$ vector field $X$ is equivalent to all vector fields in a neighborhood $U$ of $X$ in the $C^1$ topology, then the equivalence homeomorphism between $X$ and any vector field in $U$ can be chosen sufficiently close to the identity map.

Denoted by  $\Omega_{pqm}$ the set of all vector fields in $H_{pqm}$ which are structurally stable with respect to perturbations in $H_{pqm}$.
In this paper we characterise the vector fields in $\Omega_{pqm}$  and determine the exact number of the topological equivalence classes in $\Omega_{pqm}$. By using this characterisation we  give an extension of The Hartman-Grobman Theorem for a kind of planar vector fields.

We would like to  point out that the results of this paper  are an extension of the results in \cite{LRR} and an improvement of them. For a given  $(p,q)$-quasihomogeneous vector field of degree $m$,   we give an explicit method   to decide whether  it is structurally stable  with respect to perturbation in $H_{pqm}$, before finding the vector field induced by $X$ in the Poincar\'e-Lyapunov sphere, as it has taken place in others papers such as \cite{LRR}, for instance.  Roughly speaking,  if both $p$ and $q$ are odd, the study of quasihomogeneous system are similar to those of  the homogeneous one  in  \cite{LRR};    if either $p$ or $q$ is odd while the other is even,  then we should use different ideas to  study the $(p,q)$-quasihomogeneous vector fields. In fact, we present some  results which do not appear in  \cite{LRR} for the latter case. For example,  one of the   interesting results is that there are triples $(p,q,m)$ such that $H_{pqm}\not=\emptyset$ but $\Omega_{pqm}=\emptyset$, which does not occur in the homogeneous case. We  also note that Proposition \ref{p5} is not necessarily in \cite{LRR}, but it is crucial for our analysis. Lemma \ref{l4.1} will tell us what is the difference between a general quasihomogeneous vector field and a homogeneous one.

\medskip

Now we shall present the main results of this paper.  Let
\begin{equation}\label{eq2}
\eta(x,y)=pxQ(x,y)-qyP(x,y).
\end{equation}

In Lemma \ref{l9} (see Section \ref{s2} below)  it is proved that if $\eta(x,y) \equiv 0$, then $(P,Q)=(px,qy)$. This degenerated case will not be considered in this paper. Moreover we understand that the following convention (a) holds up  without loss of generality. See Lemma \ref{l1} in Section \ref{s2}.
\medskip

\noindent {\bf Convention.} {\it In this paper we understand that
\begin{itemize}
 \item[(a)]$p$ is odd and $(p,q)=1$, unless the opposite is claimed, and
\item[(b)] $\eta(x,y)\not\equiv 0$.
\end{itemize}}

\medskip

In Section \ref{s2} we shall study the phase portraits of the vector fields in $H_{pqm}$ and prove the following theorem.

\begin{theorem}\label{th3}
Let $X\in H_{pqm}\not=\emptyset$. If  $\eta(1,y)$ has no zero and $\eta(0,1)\not=0$, then
$\eta(x,y)\not=0$, for $(x,y)\not=(0,0)$ and, the origin of system (\ref{eq1}) is
\begin{itemize}
\item[(a)] a global center if and only if $I_X=0$, where
\begin{equation}\label{eq3}
I_X=\int_{-\infty}^{+\infty}\frac{P(1,u)}{\eta(1,u)}du.
\end{equation}
\item[(b)] a global stable (respectively, unstable) focus if and only if
$I_X<0$ (respectively, $I_X>0$).
\end{itemize}
\end{theorem}

In \cite{LLYZ} $I_X$ is given by the integral of $(p,q)$-trigonometric functions.
Here $I_X$ depends on the integral of the rational function $P(x,y)/\eta(x,y)$.

In Section \ref{s2} it is proved that, if $\eta(1,y)=0$ has a real zero  or $\eta(0,1)=0$  then the system \eqref{eq1} has at least
one invariant curve. Therefore, Theorem \ref{th3} implies that the origin is a center if and only if
$\eta(x,y)\not=0,\,\,I_X=0$, or if and only if $\eta(1,y)\not=0,\,\,\eta(0,1)\not=0,\,\,I_X=0$.

\medskip

The following theorem is proved in Section \ref{s3} and it gives the characterisation of the set $\Omega_{pqm}$.

\begin{theorem}\label{th4}
Denoted by  $\Omega_{pqm}$ the set of  all vector fields in $H_{pqm}$ which are structurally stable with respect to perturbations in $H_{pqm}\not=\emptyset$. If $ \Omega_{pqm}\not=\emptyset$, then the vector field $X \in \Omega_{pqm}$ if and only if one of the following conditions is satisfied:
\begin{itemize}
\item[(a)] If $\eta(1,y)$ has no zero and $\eta(0,1)\not= 0$ then $I_X\not=0$;
\item[(b)] If the condition in (a) does not hold, then  all the zeros of $\eta(1,y)$ are simple if they exist, and $\partial\eta(0,1)/\partial x\not=0$, if $\eta(0,1)=0$.
\end{itemize}
\end{theorem}

In order to compute the number of topological equivalence classes in $\Omega_{pqm}$ it is necessary to have some additional information about the normal form of the structural stable vector field $X=(P,Q)$. They are given by the following proposition.

\begin{proposition}\label{p5}
Let $X =(P,Q) \in H_{pqm}\not=\emptyset$ and $\eta(x,y)$  is given by \eqref{eq2}. Then  $\eta(x,y)$ is a $(p,q)$-
quasihomogeneous polynomial of degree $p+q+m-1$ having the form
\begin{equation}\label{eq4}
\eta(x,y)=\sum_{pi+qj=p+q+m-1}c_{ij}x^iy^j.
\end{equation}
Suppose  $X \in \Omega_{pqm}\not=\emptyset$, then there exist a unique integer $r$ such that
\begin{itemize}
\item[(a)] if  $\eta(0,1)\not=0,\,\eta(1,0)\not=0$ then
$p+q+m-1=(r+1)pq$ with $r\geq 0$ and
\begin{equation}\label{eq5}\begin{array}{ccl}
\eta(x,y)&=&\displaystyle {\sum_{l=0}^{r+1}c_{lq,(r+1-l)p}(x^q)^l(y^p)^{r+1-l}},\\[2ex]
P(x,y)&=&\displaystyle {\sum_{l=0}^ra_{lq,(r+1-l)p-1}x^{lq}y^{(r+1-l)p-1}},\\[2ex]
 Q(x,y)&=&\displaystyle {\sum_{l=0}^rb_{(r+1-l)q-1,lp}x^{(r+1-l)q-1}y^{lp}},\end{array}
\end{equation}
with $a_{0,(r+1)p-1}\not=0,\,\,b_{(r+1)q-1,0}\not=0,\,\,c_{0,(r+1)p}\not=0,\,\,c_{(r+1)q,0}\not=0$;
\vspace*{.4cm}

\item[(b)] if $\eta(0,1)\not=0,\,\eta(1,0)=0$, then $p+m-1=rpq$ with $r\geq 0$ and
\begin{equation}\label{eq6}\begin{array}{ccl}
\eta(x,y)&=&\displaystyle {\left(\sum_{l=0}^rc_{lq,(r-l)p+1}(x^q)^l(y^p)^{r-l}\right)y},\\[2ex]
P(x,y)&=&\displaystyle {\sum_{l=0}^ra_{lq,(r-l)p}(x^q)^l(y^p)^{r-l}},\\[2ex]
 Q(x,y)&=&\displaystyle {\sum_{l=0}^rb_{(r-l)q-1,lp+1}x^{(r-l)q-1}y^{lp+1}},\end{array}
\end{equation}
with $a_{0,rp}\not=0,\,\,b_{rq-1,1}\not=0,\,\,c_{0,rp+1}\not=0,\,\,c_{rq,1}\not=0$;
\vspace*{.4cm}

\item[(c)]if $\eta(0,1)=0,\,\eta(1,0)\not=0$, then $q+m-1=rpq$  with $r\geq 0$ and
\begin{equation}\label{eq7} \begin{array}{ccl}
\eta(x,y)&=&\displaystyle{x\left(\sum_{l=0}^rc_{1+lq,(r-l)p}(x^q)^l(y^p)^{r-l}\right)},\\[2ex]
P(x,y)&=&\displaystyle{x\left(\sum_{l=0}^{r-1}a_{1+lq,(r-l)p-1}x^{lq}y^{(r-l)p-1}\right)},\\[2ex]
Q(x,y)&=&\displaystyle{\sum_{l=0}^rb_{(r-l)q,lp}(x^q)^{r-l}(y^p)^l},\end{array}
\end{equation}
with $a_{1,rp-1}\not=0,\,\,b_{rq,0}\not=0,\,\,c_{1,rp}\not=0,\,\,
c_{1+rq,0}\not=0$;
\vspace*{.4cm}

\item[(d)]if $\eta(0,1)=0,\,\eta(1,0)=0$, then $m-1=(r-1)pq$ with $r\geq 1$ and
\begin{equation}\label{eq8} \begin{array}{ccl}
\eta(x,y)&=&\displaystyle{xy\left(\sum_{l=0}^{r-1}c_{1+lq,(r-1-l)p+1}(x^q)^l(y^p)^{r-1-l}\right)},\\[2ex]
P(x,y)&=&\displaystyle{x\left(\sum_{l=0}^{r-1}a_{1+lq,(r-1-l)p}(x^q)^l(y^p)^{r-1-l}\right)},\\[2ex]
Q(x,y)&=&\displaystyle{y \left(\sum_{l=0}^{r-1}b_{(r-1-l)q,1+lq}(x^q)^{r-1-l}(y^p)^l\right)},\end{array}
\end{equation}
with $a_{1,(r-1)p}\not=0,\,\,b_{(r-1)q,1}\not=0\,\,c_{1,(r-1)p}\not=0,\,\,
c_{1+(r-1)q,1}\not=0$.
\end{itemize}
\end{proposition}

\bigskip

Proposition \ref{p5} says that, if $X \in \Omega_{pqm}$ then there is a non-negative integer $r$ such that $r$  satisfies at least one of the
following equations:
\begin{equation}\label{eq9}
\begin{array}{lll}
\mathcal{E}_1: p+q+m-1=(r+1)pq, & & \mathcal{E}_2: p+m-1=rpq,\\
\mathcal{E}_3: q+m-1=rpq, & & \mathcal{E}_4: m-1=(r-1)pq.
\end{array}
\end{equation}

An interesting  consequence of the above  proposition  is the following: there are $p,q$ and $m$ such that $H_{pqm}\not=\emptyset$ but $\Omega_{pqm}=\emptyset$. This  is  proved by  choosing the triple  $(p,q,m)$ such that there is no $r$ satisfying any equation  in \eqref{eq9}.

\medskip
\noindent {\bf Example.} Let $(p,q,m)=(1,7,2)$ and consider $H_{172}$. Since $X=(ax^2,bx^8+cxy)\in H_{172}$ for each $(a,b,c)\in \mathbb{R}^3$, we have $H_{172}\not=\emptyset$. It is easy to check that there is no $r$ satisfying any equation  in \eqref{eq9} then $\Omega_{172}=\emptyset$.
\medskip

Denote by
\begin{equation}\label{eq10}
\Theta_i=\{(p,q,m): \mbox{ there exists } r \mbox{ satisfying } \mathcal{E}_i \}.
\end{equation}

We shall classify $\Omega_{pqm}$ by $\Theta_i,\,\,i=1,2,3,4$, and  study the number of equivalence classes in $\Omega_{pqm}$.  The following lemma shows what  happens if $(p,q,m)\in  \Theta_i\cap \Theta_j,\,j\not=i$. It also tells us that if $p=q=1$ (the homogeneous case), then $(1,1,m)$ satisfies all equations in \eqref{eq9}.

\begin{lemma}\label{l4.1} Let $r_i$ be solution of $\mathcal{E}_i$.  If $(p,q,m) \in \Theta_i \cap \Theta_j,\,\,i\not=j$, then $r_i=r_j$. More precisely,
\begin{itemize}
\item[(a)]  $p=q=1$ if one of the following condition holds:
\begin{itemize}
\item[(a.1)] $(p,q,m)\in \Theta_1 \cap \Theta_4$;
\item[(a.2)] $(p,q,m)\in \Theta_2 \cap \Theta_3$;
\item[(a.3)] $(p,q,m)$ satisfies three equations of \eqref{eq9};
\end{itemize}
\item[(b)] $p=1$  if one of the following two  condition holds:
\begin{itemize}
\item[(b.1)] $(p,q,m)\in \Theta_1 \cap \Theta_2$;
\item[(b.2)] $(p,q,m)\in \Theta_3 \cap \Theta_4$;
\end{itemize}
\item[(c)] $q=1$  if one of the following two  condition holds:
\begin{itemize}
\item[(b.1)] $(p,q,m)\in \Theta_1 \cap \Theta_3$;
\item[(b.2)] $(p,q,m)\in \Theta_2\cap \Theta_4$.
\end{itemize}
\end{itemize}
If $p=1,\,\,(1,q,m)\in \Theta_1 \cap \Theta_2 $(resp. $\Theta_3 \cap \Theta_4$), $X=(P(x,y),Q(x,y))\in \Omega_{1qm}\not=\emptyset$, then  $X$ with \eqref{eq6} (resp. \eqref{eq8}) can be changed into system \eqref{eq1}  with \eqref{eq5} (resp. \eqref{eq7}) by the $(1,q)$-quasihomogeneous polynomial transformation  $y\rightarrow y-\lambda x^q$, $0\not=\lambda\in \mathbb{R}$.
\end{lemma}

It follows from the above lemma that for given $p,\,q$ and $m$, if there exists the number $r$, then  it is unique and  can not take different values, even if $(p,q,m,r)$ satisfies two or more equations of \eqref{eq9}. It also says that there exists a triple $(p,q,m)$ such that $(p,q,m) \in \Theta_i $ but  $(p,q,m) \not\in \Theta_j, \,\, j \not=i$.

\medskip

Based on Lemma \ref{l4.1}, without loss of generality, we assume that the following conventions hold, provided  $X\in \Omega_{pqm}\not=\emptyset$:
\begin{itemize}
\item[(i)]  $p=1$, if $p=1$ or $q=1$;
\item[(ii)]  In order to compute the number of topological equivalence classes in $\Omega_{pqm}$, we consider the case   $(1,q,m)\in \Theta_1$ (resp. $\Theta_3$)  if $(1,q,m)\in \Theta_1 \cap \Theta_2$ (resp. $(1,q,m)\in \Theta_3 \cap \Theta_4$).
\end{itemize}

Denote by $C_{pqm}$  the number of topological equivalence classes in $\Omega_{pqm}$. The following theorem gives us the value of $C_{pqm}$ for each $p, q$ and $m$.

\begin{theorem}\label{th6}
Let  $r$  be an integer, defined in  Proposition \ref{p5}, and    $\Omega_{pqm}\not=\emptyset$. Assume that the above conventions (i) and (ii) hold.
\begin{itemize}
\item[(a)]Suppose that both $p$ and $q$ are odd.
\begin{itemize}
\item[(a.1)] If $r$ is odd, then
\begin{equation*}
C_{pqm}=\left\{\begin{array}{ll} 1+\dfrac{1}{2}\displaystyle{\sum_{j=1}^{(r+1)/2}\sum_{n|j}({\mathcal P}_{2n}+I_{2n})},&\mathrm{if\,\,}\displaystyle{(p,q,m)\in\Theta_1\backslash  \bigcup_{i=2}^4}\Theta_i,\\[2ex]
\displaystyle{-1+\dfrac{1}{2}\sum_{j=1}^{(r+1)/2}\sum_{n|j}({\mathcal P}_{2n}+I_{2n})},&\mathrm{if\,\,}\displaystyle{(p,q,m)\in\bigcup_{i=2}^4\Theta_i},
\end{array}\right.
\end{equation*}
where
\begin{equation*}
\mathcal{P}_{2n}=\frac{1}{n}\left(2^{2n}-\sum_{l|n,l\not=n}l\mathcal{
P}_{2l}\right),\,\,I_{2n}=2^{n+1}-\sum_{l|n,l\not=
n}I_{2l},
\end{equation*}

\item[(a.2)] If $r$ is even, then
\begin{equation*}
C_{pqm}=-1+\dfrac{1}{2}\sum_{j=1}^{r/2}\sum_{n|2j+1}({\mathcal P}_{2n}+I_{2n}),
\end{equation*}
where \begin{equation}\label{eq11}
\mathcal{P}_{2n}=\frac{1}{n}\left(2^n-\sum_{l|n,l\not=n}l\mathcal{
P}_{2l}\right),\,\,I_{2n}=2^{(n+1)/2}-\sum_{l|n,l\not=
n}I_{2l},
\end{equation}
\end{itemize}

\item[(b)] Suppose that $p$ is odd and $q$ are even and ${\mathcal P}_{2n}$ is given by \eqref{eq11}.
\begin{itemize}
\item[(b.1)] If  $\displaystyle{(p,q,m)\in\Theta_1\backslash\bigcup_{i=2}^4\Theta_i}$, then
\begin{equation*}
C_{pqm}=\left\{\begin{array}{ll} \displaystyle{\dfrac{1}{2}\left(r+3+\sum_{j=1}^{(r+1)/2}\sum_{n|2j}{\mathcal P}_{2n}\right)},&\mathrm{if\,\,}r\mathrm{\,\,is\,\, odd},\\[2.5ex]
\displaystyle{\dfrac{1}{2}\left(r-2+\sum_{j=1}^{r/2}\sum_{n|2j+1}{\mathcal P}_{2n}\right)},&\mathrm{if\,\,}r\mathrm{\,\,is \,\,even},
\end{array}\right. .
\end{equation*}
\item[(b.2)] If $(p,q,m) \in \Theta_2$, then
\begin{equation*}
C_{pqm}=\left\{\begin{array}{ll} \displaystyle{\dfrac{1}{2}\left(r-1+\sum_{j=1}^{(r+1)/2}\sum_{n|2j}{\mathcal P}_{2n}\right)},&\mathrm{if\,\,}r\mathrm{\,\,is\,\, odd},\\[2ex]
\displaystyle{\dfrac{1}{2}\left(r-2+\sum_{j=1}^{r/2}\sum_{n|2j+1}{\mathcal P}_{2n}\right)},&\mathrm{if\,\,}r\mathrm{\,\,is \,\,even},
\end{array}\right. .
\end{equation*}

\item[(b.3)] If $(p,q,m) \in \Theta_3\cup \Theta_4$, then
\begin{equation*}
C_{pqm}=\left\{\begin{array}{ll}\displaystyle{r+ \dfrac{1}{2}\left(\sum_{j=1}^{(r+1)/2}\sum_{n|j}{\mathcal P}_{2n}\right)},&\mathrm{if\,\,}r\mathrm{\,\,is\,\, odd},\\[2ex]
\displaystyle{\dfrac{1}{2}\left(r-2+\sum_{j=1}^{r/2}\sum_{n|2j+1}{\mathcal P}_{2n}\right)},&\mathrm{if\,\,}r\mathrm{\,\,is \,\,even},
\end{array}\right. .
\end{equation*}
\end{itemize}
\end{itemize}
\end{theorem}

\medskip

Finally we have an extension of the Hartman-Grobman Theorem to vector fields in $H_{pqm}$. Before stating these results we need some notations.
We say that two analytic vector fields $X$ and $Y$ are {\it locally topologically equivalent at origin (resp. infinity)}
if there are two neighborhoods $U$ and $V$ of the origin (resp. the infinity) and a homeomorphism $h: U \rightarrow V$
that carries orbits of the flow induced by $X$ onto orbits of the flow induced by $Y$, preserving sense but not necessarily parametrization (see for instance \cite{LRR}).

  The next two theorems are extensions of the Hartman-Grobman Theorem at the origin and infinity respectively.

\begin{theorem} \label{th7} Let $\displaystyle {X=\sum_{i\geq m}X_i}$, where $X_i=(P_i(x,y),Q_i(x,y))$ is a $(p,q)$-quasi\-ho\-mo\-ge\-neous polynomial vector field of degree $i$ if $H_{pqi}\not=\emptyset$, and $X_i=(0,0)$ if $H_{pqi}=\emptyset$, where $i \in \{m,m+1,\cdots,\}$, $m\geq 1$. Suppose $X_m \in \Omega_{pqm}$, then the phase portrait of $X$ and $X_m$ are locally topologically equivalent at the origin.
\end{theorem}

\medskip

Now we can get the analogous of Theorem \ref{th7} at infinity.

\begin{theorem}\label{th8}
Let $\displaystyle {\hat X}=\sum_{i=0}^{m} X_i$,  where $X_i=(P_i(x,y),Q_i(x,y))$ is a $(p,q)$-qua\-siho\-moge\-neous polynomial vector field of degree $i$ if $H_{pqi}\not=\emptyset$, and $X_i=(0,0)$ if $H_{pqi}=\emptyset$, where $i \in \{ 0,1,... ,m\} $. Suppose $X_m \in \Omega_{pqm}$, then the phase portraits
of  $X_m$ and ${\hat X}$ are locally topologically equivalent at infinity. \end{theorem}

\medskip

The rest of this paper  is organized as follows. In Section \ref{s2} we study the phase portraits of
$(p,q)$-quasihomogeneous polynomials of degree $m$ on Poincar\'e-Lyapunov disk. In Section \ref{s3}
we prove Theorem \ref{th4} which characterises the vector fields $X\in \Omega_{pqm}$. In Section \ref{s4} we compute
the number $C_{pqm}$ of topological equivalence classes in $\Omega_{pqm}$.  Theorem \ref{th7} and Theorem \ref{th8} are proved in Section \ref{s5}.

\section{Phase portraits of quasihomogeneous vector fields}\label{s2}

In this section we shall study the phase portraits of quasihomogeneous vector fields.

 As discussed before, one of our conventions is that, if $X\in H_{pqm}$ then  $(p,q)=1$. The next lemma shows that, if $(p,q)\geq 2 $, then there exist a triple $(p',q',m')$ with $(p',q')=1$ such that  $X \in H_{p'q'm'}$.

\begin{lemma}\label{l1}
Suppose $(p,q) = k \geq 2$ in \eqref{eq1}, then there exists a unique triple $(p',q',m')$ with $(p',q')=1$ such that  system \eqref{eq1} is a
$(p',q')$-quasihomogeneous vector field of degree $m'$.
\end{lemma}
\begin{proof}
Let $p=kp',\,q=kq'$. Since $(p,q) = k \geq 2$, we have  $(p',q')=1$.  It follows from the definition of  $P(x,y)$ in \eqref{eq1} that $k p'- 1 + m = k p' i + k q'j$, which implies  that $k|m-1$. Taking $m'=1+(m-1)/k$, the statement follows.
\end{proof}

\begin{lemma}\label{l9} Let $X=(P,Q) \in H_{pqm}\not=\emptyset$ and consider the polynomial $\eta(x,y)$, defined by \eqref{eq2}.
\begin{itemize}
\item[(a)] $\eta(x,y)$ is a $(p,q)$-quasihomogeneous polynomial of degree $p+q+m-1$.
\item[(b)]If $\eta(x,y)\equiv 0$, then $X=(P,Q)=(px,qy)$.
\end{itemize}
\end{lemma}
\begin{proof} The statement (a) is a consequence of the fact that $P(x,y)$ and $Q(x,y)$ are $(p,q)$-quasihomogeneous polynomials of degree $m-1+p$ and $m-1+q$, respectively. Now if  $\eta(x,y)\equiv 0$, then $pxQ(x,y)=qyP(x,y)$. As  $P(x,y)$ and $Q(x,y)$ are coprime, follows that
 $P(x,y)|x$ and $Q(x,y)|y$. The result (b) follows.
\end{proof}

The lemma bellow says that, as in the homogeneous case, there is no  limit cycles if  $\eta(1,y)=0$  has zeros, or $\eta(0,1)=0$.

\begin{proposition}\label{p10}
Let $X\in H_{pqm}\not=\emptyset$.
\begin{itemize}
\item[(a)] If $\eta(0,1)=0$, then $x=0$ is an invariant line for the flow of $X$;
\item[(b)] If there exists $\lambda \in \mathbb R$ such that $\eta(1,\lambda)=0$, then  $y^p-\lambda^p x^q=0$ is an invariant curve
of the flow of $X$;
\item[(c)] $X$ has no periodic orbit if (a) or (b) holds.
\end{itemize}
\end{proposition}

\begin{proof}
As $P(x,y)$ is $(p,q)$-quasihomogeneous polynomial of degree $p+m-1$ it follows that $P(0,y^q)=y^{p-1+m}P(0,1)=0$, if $0=\eta(0,1)=P(0,1)$.
Therefore $x=0$ is an invariant line of $X$.

Suppose there exists $\lambda$ such that $\eta(1,\lambda)=0$. If $\lambda=0$, then $0=\eta(1,0)=pQ(1,0)$. So
$Q(x^p,0)=x^{q-1+m}Q(1,0)=0$. Therefore $y=0$ is an invariant curve of $X$. Otherwise, let $V(x,y)=y^p-\lambda^px^q$. Then
\begin{eqnarray*}
\frac{dV}{dt}\Big|_{y=\lambda
x^{q/p}}&=&p\lambda^{p-1}x^{q(p-1)/p}Q(x,\lambda
x^{q/p})-q\lambda^px^{q-1}P(x,\lambda x^{q/p})\\
&=&\lambda^{p-1}(x^{1/p})^{pq+m-1}(pQ(1,\lambda)-q\lambda P(1,\lambda))\\
&=&\lambda^{p-1}(x^{1/p})^{pq+m-1}\eta(1,\lambda)=0,
\end{eqnarray*}
where  $x^{q/p}:=(x^{1/p})^q$. This ends the proof of (b).

From \cite{LLYZ}  system \eqref{eq1} has  a unique singular
point at the origin. If there exists an invariant curve passing through
the unique singular point of system \eqref{eq1}, then no limit
cycle can surround the origin and the statement (c) is proved.
\end{proof}

To study the singular point $(0,0)$ of (\ref{eq1}), we
introduce the {\it $(p,q)$-trigonometric functions
$z(\phi)=\mathrm{Cs}\phi$ and $\omega(\phi)=\mathrm{Sn}\phi$}
\cite{L,LLYZ} as the solution of the following initial problem
\begin{equation*}
\dot
z=-\omega^{2p-1},\,\,\dot\omega=z^{2q-1},\,\,z(0)=p^{-\frac{1}{2q}},\,\,\omega(0)=0.
\end{equation*}
It is  know that ${\mathrm Cs}\phi$ and ${\mathrm Sn}\phi$ are
$\mathcal T$-periodic functions with
\begin{equation*}
{\mathcal
T}=2p^{-\frac{1}{2q}}q^{-\frac{1}{2p}}\frac{\Gamma(\frac{1}{2p})\Gamma(\frac{1}{2q})}{\Gamma(\frac{1}{2p}+\frac{1}{2q})}.
\end{equation*}
and satisfies
\begin{equation*}
p{\mathrm Cs}^{2q}\phi+q{\mathrm Sn}^{2p}\phi=1,\,\,\frac{d{\mathrm
Cs}\phi}{d\phi}=-{\mathrm Sn}^{2p-1}\phi,\,\,\frac{d{\mathrm
Sn}\phi}{d\phi}={\mathrm Cs}^{2q-1}\phi.
\end{equation*}
For $(p,q)=(1,1)$, we have that ${\mathrm Cs}\phi=\cos
\phi,\,\,{\mathrm Sn}\phi=\sin\phi$, i.e. the $(1,1)$-trigonometric
functions are the classical ones.

In the {\it $(p,q)$-polar coordinates $(r,\phi)$}
\begin{equation}\label{eqpq}
x=r^p{\mathrm Cs}\phi,\,\,\,y=r^q{\mathrm Sn}\phi,
\end{equation}
the planar $(p,q)$-quasihomogeneous  system \eqref{eq1} of degree $m$ is written as
\begin{equation*}
\dot r=r^mF(\phi),\,\,\,\,\dot \phi=r^{m-1}G(\phi),
\end{equation*}
with
\begin{equation*}
F(\phi)=\xi\left({\mathrm Cs}\phi,{\mathrm Sn}\phi\right),\,\,\,\,\,\,\,G(\phi)=\eta\left({\mathrm Cs}\phi,{\mathrm Sn}\phi\right),
\end{equation*}
where  $\eta(x,y)$ is defined in \eqref{eq2} and
 \begin{equation}\label{eq15}
 \xi(x,y)=x^{2q-1}P(x,y)+y^{2p-1}Q(x,y).
 \end{equation}

Taking the change of coordinates
\begin{equation}\label{eqst}
\frac{ds}{dt}=r^{m-1},\,\,\,\theta=\frac{2\pi\phi}{\mathcal T},
\end{equation}
the above system goes over to
\begin{equation}\label{eq16}
r'=rf(\theta),\,\,\,\, \theta'=g(\theta).
\end{equation}
where prime denotes derivative with respect to $s$,
\begin{equation}\label{eq17}
f(\theta)=F\left(\frac{{\mathcal
T}\theta}{2\pi}\right),\,\,\,g(\theta)=\frac{2\pi}{\mathcal
T}G\left(\frac{{\mathcal T}\theta}{2\pi}\right).
\end{equation}

It is easy to check that $f(\theta)$ and $g(\theta)$ are
$2\pi$-periodic functions.

\medskip

Now, we are able to prove  Theorem \ref{th3} which discuss about the phase portrait of the vector fields in $H_{pqm}$, provided that $\eta(1,y)$ has no zero and $\eta(0,1)\not=0$.

\begin{proof}[Proof of Theorem \ref{th3}]
First we shall prove that $\eta(x,y)\not=0$ for all $(x,y)\not=(0,0)$.

From Lemma \ref{l9}(a) it follows that if $x \neq 0$, then
\begin{equation}\label{eq14}
\eta(x,y)=(x^{1/p})^{p+q-1+m}\eta\left(1,\frac{y}{x^{q/p}}\right),
\end{equation}
As $\eta(1,y)$ has no zero, we get $\eta(x,y)\not=0$  from \eqref{eq14}, provided that $x \neq 0$.

In what follows we shall prove that $\eta(0,y) \neq 0$. Suppose it does not happen and there exists $y^*\neq 0$ such that $0= \eta(0,y^*) = -q y^* P(0,y^*)$. So
\begin{equation*}
0=P(0,y^*)=(|y^*|^{1/q})^{p-1+m}P(0, {\rm sgn} (y^*)).
\end{equation*}
where
\begin{equation*}
0=P(0,{\rm sgn} (y^*))=\left\{\begin{array}{cl}P(0,1),&\mathrm{if}\,\, y^*>0.\\[2ex]
P(0,-1),&\mathrm{if}\,\, y^*<0.\end{array}\right.
\end{equation*}
If $y^*<0$, then for any $y < 0$,
\begin{equation*}
P(0,y)=(|y|^{1/q})^{p-1+m}P(0, -1)\equiv 0.
\end{equation*}
As $P(0,y)$ is a polynomial in the variable $y$ and it is identically zero for  $y < 0$,  we have $P(0,y)\equiv 0$ and  $P(0,1)=0$.
Hence we conclude that if there exists $y^*\neq 0$ such that $\eta(0,y^*)=0$, then $\eta(0,1)= -q P(0,1)= 0$. This is a contradition with the hypotheses of this theorem. Therefore, $\eta(0,y)\neq 0$ for $y \neq 0$ and one obtain $\eta(x,y)\neq 0$ for $(x,y)\neq (0,0)$.

Secondly  we shall prove the statements (a) and (b).  If $\eta(x,y)\neq 0$, then $g(\theta)\not=0$ (see \eqref{eq17}). It follows from \eqref{eq16} that
\begin{equation*}
\frac{dr}{d\theta}=\frac{rf(\theta)}{g(\theta)},
\end{equation*}
and the first return map is given by
\begin{equation*}
r(2\pi, r_0)=\left\{\begin{array}{ll}r_0e^{\tilde {I}_X}, & \mathrm{if\,\,\,}f(\theta)\not\equiv 0,\\[2ex]
 r_0,& \mathrm{if\,\,\,}f(\theta)\equiv 0,\end{array}\right.
\end{equation*}
where
\begin{equation*}
{\tilde I}_X=\int_0^{2\pi}\frac{f(\theta)}{g(\theta)}d\theta,
\end{equation*}
and $(r_0,0)$ is one point of the positive $x$-axis. From the first return map we
deduce that if ${\tilde I}_X=0$, then the
origin is a center. If ${\tilde I}_X < 0$ (resp. ${\tilde I}_X > 0$), then it is a stable (resp. unstable)  focus.

To end this proof we must show that the sign of $I_X$ is equal to
the sign of ${\tilde I}_X$.

Let $x=\mathrm{Cs}\phi,\,\,y=\mathrm{Sn}\phi$. Then
\begin{equation*}
{\tilde
I}_X=\frac{2\pi}{\mathcal{T}}\int_{0}^{\mathcal{T}}\frac{F(\phi)}{G(\phi)}d\phi
=\frac{2\pi}{\mathcal{T}}\oint_{px^{2q}+qy^{2p}=1}\frac{\xi(x,y)}{x^{2q-1}\eta(x,y)}dy
=\frac{2\pi}{\mathcal{T}}({\tilde I}_+-{\tilde I}_-),
\end{equation*}
where
\begin{equation*}
{\tilde
I}_\pm=\int_{-q^{-1/(2p)}}^{q^{-1/(2p)}}\frac{\xi(x,y)}{x^{2q-1}\eta(x,y)}\Big|_{x=x_{\pm}=\pm
p^{-1/(2q)}(1-qy^{2p})^{1/(2q)}} dy.
\end{equation*}
If  $p$ is odd and $x\not=0$,  then
\begin{equation}\label{eq18}
\xi(x,y)=(x^{1/p})^{2pq-1+m}\xi\left(1,\frac{y}{x^{q/p}}\right),
\end{equation}
where $\xi(x,y)$ is defined in \eqref{eq15}. It follows from \eqref{eq14} and \eqref{eq18} that
\begin{equation*}
{\tilde
I}_+=\int_{-q^{-1/(2p)}}^{q^{-1/(2p)}}\frac{\xi(1,\frac{y}{x^{q/p}})}{x^{q/p}\eta(1,y/x^{q/p})}\Big|_{x=x_+}
 dy.
\end{equation*}
Let $x=x_\pm$ and $u=y/x^{q/p}$. Then $u^{2p}=y^{2p}/x^{2q}$ and $y^{2p}=u^{2p}/(p+qu^{2p})$. Hence

\begin{eqnarray*} {\tilde
I}_+&=&\int_{-\infty}^{\infty}\frac{\xi(1,u)}{x^{q/p}\eta(1,u)}\Big|_{x=x_+}
\frac{1}{2py^{2p-1}}\frac{d}{du}\left(\frac{u^{2p}}{p+qu^{2p}}\right)du\\
&=&\int_{-\infty}^{\infty}\frac{p\xi(1,u)}{(p+qu^{2p})\eta(1,u)}du.
\end{eqnarray*}

When $p$ is odd, \eqref{eq14} and \eqref{eq18} are also true for
$x=x_-$.  Using the same arguments as above, we get ${\tilde I}_-=-{\tilde
I}_+$ and hence ${\tilde I}_X=2{\tilde I}_+$.

As
\begin{equation*}
\int_{-\infty}^{\infty}\frac{p\xi(1,u)}{(p+qu^{2p})\eta(1,u)}du-\int_{-\infty}^{\infty}\frac{P(1,u)}{\eta(1,u)}du\\
=\int_{-\infty}^{\infty}\frac{u^{2p-1}}{p+qu^{2p}}du=0,
\end{equation*}
the statements (a) and (b) follows, provided that $p$ is odd. \end{proof}

\medskip

Our next step is to consider the case where $\eta(1,y)$ has zeros or $\eta(0,1)$ is zero. To this we consider the study of the finite and infinity singularities of $X$.

To study with more details the singularities at infinity, we use {\it Poincar\'e-Lyapunov Compactification}, see for instance \cite{DLA}.
First we blow up the system \eqref{eq1} in the positive $x$-direction by \begin{equation*}
x=\frac{1}{z^p},\,\,\,y=\frac{u}{z^q},
d\tau=\frac{z^{1-m}dt}{p},\,\,z>0.
\end{equation*}
This yields the vector field
\begin{equation}\label{eq22}
\frac{dz}{d\tau}=-zP(1,u),\,\,\,\,\frac{du}{d\tau}=\eta(1,u).
\end{equation}
 Each point $(z,u)=(0,\lambda)$ satisfying $\eta(1,\lambda)=0$ is an infinity
singularity of system \eqref{eq1} and has its linear
part given by
\begin{equation}\label{eq23}
\left(\begin{array}{cc} -P(1,\lambda) & 0\\
0&\frac{\partial \eta}{\partial u}(1,\lambda)
\end{array}
\right).
\end{equation}
Moreover, if $P(1,\lambda)=0$, then $0=\eta(1,\lambda)=p Q(1,\lambda)$, so $P(x,\lambda x^{q/p})=Q(x,\lambda x^{q/p})\linebreak =0$ for
$x>0$. Hence the algebraic curves $P(x,y)=0$ and $Q(x,y)=0$ have
infinite intersection points. As $P(x,y)$ and $Q(x,y)$ are
coprime polynomials, it follows from B\'ezout theorem \cite{Fu} that
$P(x,y)=0$ and $Q(x,y)=0$ have finite intersection points on $\mathbb{R}^2$ , what is a contradiction with our hypothesis. Therefore,
$P(1,\lambda)\not=0$ if $\eta(1,\lambda)=0$ but $\eta(1,y)\not\equiv 0$. Hence  all singular points of system \eqref{eq22}
 are elementary.

Second we blow up the system \eqref{eq1} in the negative  $x$-direction
using the transformation
\begin{equation*}
x=-\frac{1}{z^p},\,\,\,y=\frac{u}{z^q},
d\tau=\frac{z^{1-m}dt}{p},\,\,z>0.
\end{equation*}
One gets
\begin{equation}\label{eq24}
\frac{dz}{d\tau}=zP(-1,u),\,\,\,\,\frac{du}{d\tau}=-\eta(-1,u).
\end{equation}
All points $(0,\lambda)$ satisfying $\eta(-1,\lambda)=0$ are singular points of system \eqref{eq24}.
Their linear part are given by
\begin{equation}\label{eq25}
\left(\begin{array}{cc} P(-1,\lambda) & 0\\
0&-\frac{\partial \eta}{\partial u}(-1,\lambda)
\end{array}
\right).
\end{equation}
The singular points of the system \eqref{eq24} are studied in the same way as the ones of system
\eqref{eq22}.  Moreover all of them are elementary.

Finally we blow up in the $y$-direction by
\begin{equation*}
x=\frac{v}{z^p},\,\,\,y=\pm\frac{1}{z^q},
d\tau=\frac{z^{1-m}dt}{q},\,\,z>0.
\end{equation*}
which yields two vector fields of the forms
\begin{equation}\label{eq26}
\bar X^\infty_\pm:\,\,\,\,\frac{dz}{d\tau}=\mp zQ(v,\pm
1),\,\,\,\,\frac{dv}{d\tau}=\mp \eta(v,\pm 1).
\end{equation}
 We only need to  determine whether the origin is a singular point
 of the vector fields $\bar X^\infty_\pm$. As $\eta(0,1)=0$ implies that $x=0$ is an invariant line of
$X$  (see Proposition \ref{p10}), one gets  $P(0,y)\equiv 0$ and $\eta(0,-1)=qP(0,-1)=0$.
  Using the same arguments as above, it  is
proved that $\eta(0,-1)=0$ implies $\eta(0,1)=0$. This yields that the
origin is a singular point of system $\bar X^\infty_+$ if and only if it is
a singular point of system $\bar X^\infty_-$. We also have $Q(0,\pm
1)\not=0$ if $\eta(0,\pm 1)=0$. Hence $(0,0)$ is elementary
singular point of system $\bar X^\infty_\pm$ if $\eta(0,\pm 1)=0$.

\medskip

\begin{lemma}\label{l11}
Suppose that $p$ is odd.
\begin{itemize}
\item[(a)]
$(0,\lambda)$ is  a singular points of system \eqref{eq22} if and
only if  $(0,(-1)^q\lambda)$  is a singular point of system
\eqref{eq24}.  Suppose that  $\lambda$ is a zero of
$\eta(1,\lambda)$ with multiplicity $k$, then
\begin{itemize}
\item[(i)]If $k$ is even, then
$(0,\lambda)$ and $(0,(-1)^q\lambda)$ are saddle-node of system
\eqref{eq22} and \eqref{eq24} respectively;
\item[(ii)]if $k$ is odd and $P(1,\lambda)(\partial^k \eta/\partial u^k)(1,\lambda)>0$, then
$(0,\lambda)$ and $(0,(-1)^q\lambda)$ are saddles of system
\eqref{eq22} and \eqref{eq24} respectively;
\item[(iii)] if $k$ is odd and $P(1,\lambda)(\partial^k \eta/\partial
u^k)(1,\lambda)<0$, $P(1,\lambda)>0$ (resp. $P(1,\lambda)\linebreak <0$), then
$(0,\lambda)$ is a stable  (resp. unstable) node of \eqref{eq22},
and  $(0,(-1)^q\lambda)$ is a unstable (resp. stable) node if $m$ is
even, a stable (resp. unstable) node of system \eqref{eq24} if $m$
is odd.
\end{itemize}
\item[(b)] System $\bar
X_+^\infty$ has a singular point at  $(0,0)$, if and only if system
$\bar X_-^\infty$ has    a singular point at the origin, if and
only if $\eta(0,1)=0$. Moreover, suppose  $\eta(0,1)=0$, then there
exists positive integers $k,\,l$ such that $m=(k-1)p+(l-1)q+1$, and
\begin{itemize}
\item[(i)] if $k$ is even, then   system
$\bar X_+^\infty$ (resp. $\bar X_-^\infty$) has    a saddle-node at
the origin;
\item[(ii)] if $k$ is odd and $Q(0,1)(\partial^k \eta/\partial x^k)(0,1)<0$,
then    system $\bar X_+^\infty$ (resp. $\bar X_-^\infty$) has    a
saddle at the origin;
\item[(iii)] if $k$ is odd and $Q(0,1)(\partial^k \eta/\partial
x^k)(0,1)>0$, $Q(0,1)>0$ (resp. $Q(0,1)\linebreak <0$),  then system $\bar
X_+^\infty$  has     a  stable (resp. unstable) node at the origin,
and $\bar X_-^\infty$ has a stable (resp. unstable) node at origin
if $l$ is odd, a unstable (resp. stable) node at origin if $l$ is
even, respectively.
\end{itemize}
\end{itemize}
\end{lemma}
\begin{proof}
(a) Since  $p$ is odd, it follows from \eqref{eq14} that
 \begin{equation}\label{eq27}
\eta(-1,(-1)^q \lambda)=(-1)^{p+q-1+m}\eta(1,\lambda), \end{equation} which
implies that $\eta(1,\lambda)=0$ if and only if
$\eta(-1,(-1)^q\lambda)=0$, and hence $(0,\lambda)$ is a singular
point of system \eqref{eq22},  if and only if  $(0,(-1)^q\lambda)$
is a singular point of system \eqref{eq24}. The equality
\eqref{eq27} also yields
\begin{equation}\label{eq28}
\frac{\partial^k\eta(-1,(-1)^q \lambda)}{\partial
u^k}=(-1)^{(k-1)q+m}\frac{\partial^k\eta(1, \lambda)}{\partial u^k}.
\end{equation}
Since
$P(-1,(-1)^q\lambda)=(-1)^mP(1,\lambda)$, the  statements (i),(ii)
and (iii) with $k\geq 2$ in (a) follow from  Theorem 2.19 of
 \cite{DLA}. 
If $k=1$, then all singular points are hyperbolic and
the statement (a) follows from Theorem 2.15 of the same book
\cite{DLA}.

(b) Using the same arguments as above, we get that  system $\bar X_+^\infty$ has a singular point at
$(0,0)$, if and only if  system $\bar X_-^\infty$  has  a singular
point at the origin, if and only if $\eta(0,1)=0$.

Suppose that  $\eta(x,y)$ has the form \eqref{eq4}.
If $\eta(0,1)=0$, then there are $k$ and $l$ such that
$pk+ql=p+q+m-1$, $c_{kl}\not=0$ and $c_{ij}=0$ for $i\leq k-1$,
where $k\geq 1,\,\,l\geq 1$. This gives $(\partial^k \eta/\partial
x^k)(0,\pm 1)=k! c_{kl}(\pm 1)^l$.

If $\eta(0,1)=0$, then $P(0,1)=0$, which implies that $P(x,y)$ has a
divisor $x$.  On the other hand, $Q(0,1)=0$ also means that $Q(x,y)$
has a divisor $x$. Since $P(x,y)$ and $Q(x,y)$ are coprime, we have that
$Q(0,1)\not=0$ if $\eta(0,1)=0$. The other statements in (b) follows
by the same arguments as in the proof of (a).
\end{proof}

Before ending this part some comments are necessary. If $p$ is odd,
 then it follows from \eqref{eq27} that $(0,\lambda)$ is a singular point of system
\eqref{eq22}, if and only if, $(0,(-1)^q\lambda)$ is  a
singular point of system  \eqref{eq24}. Although the information found in the positive $x$-direction also covers
the negative $x$-direction we have studied the behavior of both, positive
and negative $x$-direction, because they will be convenience  in other
proofs presented of this paper.

\medskip

Now, using similar arguments and local charts we  study  the finite singular points of the system \eqref{eq1}. It is important to observe that using $(p,q)$-polar coordinates changes one gets the vector field \eqref{eq16} which is defined on $\mathbb{S}^1 \times \mathbb{R}$.  Although the cylinder  $\mathbb{S}^1 \times \mathbb{R}$ is good surface for getting the  phase portrait near the origin, it is often less appropriate for making calculations, since we have to deal with expressions of $(p,q)$-trigonometric functions. Hence we prefer to make the calculations in different charts.

We are going to use the method of quasihomogeneous blow-up (see for instance \cite{DLA}) in local charts to study
the singular point $(0,0)$ of system \eqref{eq1}. We first blow up the vector field in the positive $x$-direction by
\begin{equation} \label{eq18-1}
x=\bar x^p,\,\,\,y=\bar x^q \bar y,\,\,d\tau=\frac{\bar
x^{m-1}dt}{p},\,\,\bar x>0,
\end{equation}
yielding
\begin{equation}\label{eq19}
\dot {\bar x}=\bar x P(1,\bar y),\,\,\,\,\dot{\bar y}=\eta(1,\bar
y).
\end{equation}
The points $(0,\lambda)$ satisfying $\eta(1,\lambda)=0$ are the
isolated singular points of \eqref{eq19} on the line $\{\bar
x=0\}$. The system in these singular points has their linear part given by
\begin{equation*}
\left(\begin{array}{cc} P(1,\lambda) & 0\\
0&\frac{\partial \eta}{\partial \bar y}(1,\lambda)
\end{array}
\right).
\end{equation*}
Since we have shown $P(1,\lambda)\not= 0$, all singular points of system \eqref{eq19} are elementary.

Next we blow  up the vector field in the negative $x$-direction, the positive $y$-direction and  the negative $y$-direction, respectively. Then we obtain three systems whose singular points are elementary.  Here we omit the details.

After blown up the origin to the elementary singular points  the singular point $(0,0)$ of
\eqref{eq1} has been desingularized. After blowing down we get the
phase portrait of system \eqref{eq1} near the origin.

\medskip

From the later considerations we have that the finite and infinite singularities of $X$ are determined by the zeros of $\eta(1,y)$ or $\eta(0,1)$. The following proposition guarantee that the invariant curves of the system \eqref{eq1} determine the phase portraits of system \eqref{eq1}.
\vskip 0.2cm

\begin{proposition}\label{p13} Suppose that $L_{\lambda_i},\,L_{\lambda_{i+1}}$ are two
consecutive invariant algebraic curve of system \eqref{eq1}, defined
in Proposition \ref{p10}. Let $S_i$ and $S_{i+1}$ be two  singular points at infinity  corresponding  to $L_{\lambda_i}$ and $L_{\lambda_{i+1}}$ respectively. Then $L_{\lambda_i}$ and $L_{\lambda_{i+1}}$ are characteristic orbits  at
$(0,0)$ and determine one sector at the origin:
\begin{itemize}
\item[(a)] two consecutive parabolic sectors at   infinity  singular points $S_i$ and $S_{i+1}$ give a
hyperbolic sector at $(0,0)$,
\item[(b)]two consecutive hyperbolic  sectors at   $S_i$ and $S_{i+1}$    give a
elliptic  sector at $(0,0)$, and
\item[(c)] one parabolic and one hyperbolic  sectors at  $S_i$ and $S_{i+1}$   give a
parabolic sector at $(0,0)$.
\end{itemize}
\end{proposition}
\begin{proof}
It follows from Lemma \ref{l11} and the study of  the phase portraits near  $(0,0)$   that  the singular points of $X$ at the origin and at  infinity   are determined by the zeros of $\eta(\pm 1,\lambda)$ or whether $\eta(0,1)=0$.  Moreover, from the Lemma \ref{l11} and Proposition \ref{p10},  if $\lambda$ is a zero of $\eta(\pm 1,\lambda)$ or $\eta(0,1)=0$, then there is an invariant curve $L$  which leaves (or going to) an infinite singular point and goes to (or leaves) the origin. So each pair of consecutive invariant algebraic curve $L_{\lambda_i}$ and  $L_{\lambda_{i+1}}$  determines a local sector  at the origin.
\end{proof}

It is important to comment that Proposition \ref{p10}, Lemma \ref{l11} and Proposition \ref{p13} are fundamental to describe the different phase portrait of the system \eqref{eq1}. They will be very important tools in Section \ref{s4} to calculate the number of topological equivalence classes in $\Omega_{pqm}$.
In fact, the phase portrait of system \eqref{eq1} is determined by the zeros of equations $\eta( 1,\lambda)=0$  and
if $\eta(0,1)=0$. Moreover, from the later proposition follows that,to describe the phase portrait of the system \eqref{eq1}, it is enough to know the sign of the product $\displaystyle P(1,\lambda)\cdot\frac{\partial \eta}{\partial y}(1,\lambda)$ for each zero of $\eta( 1,\lambda)$ and the sign of $\displaystyle P(0,1)\cdot\frac{\partial \eta}{\partial y}(0,1)$, if $\eta(0,1)=0$.

\section{Structural Stability}\label{s3}

In this section we shall prove the Theorem \ref{th4} which characterise the set $\Omega_{pqm}$ of the $(p,q)$-quasihomogeneous vector fields of degree $m$ in the plane which are structurally stable with respect to perturbations in $H_{pqm}$. As observed in the introduction we denote by $E(X)$ the induced (or extended) vector field on Poincar\'e-Lyapunov sphere and $H_{pqm}$ has the coefficient topology. Then we shall say that two vector fields $X$ and $Y$ in $H_{pqm}$ are equivalent if there exists an equivalence $h$ between the induced vector fields $E(X)$ and $E(Y)$ on the Poincar\'e-Lyapunov sphere. In our case it is redundant to required that the equivalence be near of the identity map because the Poincar\'e-Lyapunov sphere is a manifold in the Peixoto condictions and the coefficient topology is equivalent to the $C^1$ topology.

So we can start the prove of the Theorem \ref{th4}.

\begin{proof}[Proof of Theorem \ref{th4}]
Assume that $X \in H_{pqm}$ is structurally stable with respect to perturbations in $H_{pqm}$. Let us  prove that (a) and (b) happen.
If $\eta(1,y)$ has no zero and $\eta(0,1) \neq 0$, then it follows from Theorem \ref{th3} that the origin of the system \eqref{eq1}
is a global center if $I_X=0$, and a global (stable or unstable) focus if $I_X \neq 0$. If $I_X=0$  then there exists a vector field $Y$ in any neighborhood of $X$ in $H_{pqm}$ such that the phase portrait of $Y$ is a global focus. Therefore $X$ is not a structural stable vector field with respect to perturbations in $H_{pqm}$. This yields that if $X\in H_{pqm}$, then $I_X\not=0$, i.e., the condition  (a) is satisfied.

Moreover,  the origin is the unique singular point of system \eqref{eq1}\cite{LLYZ}, and the phase portrait of $X$ is completely determined by the zeros of $\eta(1,y)$ and $\eta(0,1)$ (see Lemma \ref{l11} and Proposition \ref{p13}). If $\eta(1,y)$ has a zero or $\eta(0,1) = 0$, then it follows from Proposition \ref{p10}(c) that there is no periodic orbits in the set of all solutions of the system \eqref{eq1}.

If $\eta(1,y)$ has a multiple zero then, by Lemma \ref{l11}, we conclude that there exists vector field $Y$ in a neighborhood of $X$ in $H_{pqm}$ such that $X$ and $Y$ have different number of sectors. So $X$ is not structural stable with respect to $H_{pqm}$. If $\eta(0,1)=0$ and $\partial\eta(0,1)/\partial x =0$, then  by the same arguments  we get that  $X$ is not structural stable with respect to $H_{pqm}$. Therefore, the condition (b) is satisfied.

Now, we assume that (a) or (b) are satisfied  and we shall prove that $X$ is a structurally stable vector field with respect to perturbations in $H_{pqm}$. First we suppose that (a) is true. Then there exists a neighborhood $V$ of $Y$ and a neighborhood $W$ of $X$ in $H_{pqm}$ such that if $y^* \in V$ then $\eta(1,y^*)\neq 0$ and for all $Y \in W$ we have $sign(I_X)= sign(I_Y)$. Then it follows from Theorem \ref{th3} that $X$ and $Y$ are topologically equivalent and $X$ is structurally stable with respect to perturbations in $H_{pqm}$.

 Assume that there are $s$ points $(1,y_i)$ such that $\eta(1,y_i)=0$ or $\eta(0,1)=0$ and that these zeros are simple. There exists a neighborhood $U$ of $X$ in $H_{pqm}$ such that if $Y=(P_Y,Q_Y) \in U$ then there are exactly $s$ points such that $\eta_Y(1,\lambda_i^*)=0$ or $\eta(0,1)=0$.
We can choose this neighborhood such that $sign(P \partial \eta/\partial u)(1,\lambda_i)=sign(P_Y\partial \eta_Y/\partial u)(1,\lambda_i^*))$, and  if $\eta_Y(0,1)=0$ then $sign(Q \partial \eta/\partial x)(0,1)=  sign(Q_Y\partial \eta_Y/\partial x)(0,1)$. By Lemma \ref{l11} and Proposition \ref{p13} each vector field $Y \in U$ has the same local sectors as the vector field $Y$. So they are topologically equivalent and we conclude that $X$ is structurally stable with respect to perturbations in $H_{pqm}$.
\end{proof}

Here we shall present a concrete example.

\begin{example}\label{example1}
Let $p=1$, $q=m=2$, then $H_{122}$ is non empty.  In fact if $X\in H_{122}$ then  $X$ can be written in the form $X(x,y)=(a_1 x^2 + a_2 y, a_3 x^3 + a_4 x y)$ with  $(a_1,a_2,a_3,a_4) \in {\mathbb R}^4\backslash (0,0,0,0)$. The question is, what are the conditions to $X$ to be structurally stable?  By the direct computation we have $\eta(x,y)=a_3x^4+(a_4-2a_1)x^2y-2a_2y^2$.  It follows from Theorem  \ref{th4} that $X\in \Omega_{122}$ if and only if
 the following conditions holds: (i)  $a_2 \neq 0,\,\,4 a_1^2 + a_4^2 + 4(a_1 a_4 + 2 a_2 a_3) < 0, \,\,I_X\not=0$, (ii) $a_2 \neq 0$ and $4 a_1^2 + a_4^2 + 4(a_1 a_4 + 2 a_2 a_3) > 0$.

So the vector field $X_1(x,y)=(x^2-y/2,x^3+2xy)$ and  $X_2(x,y)=(x^2-y, 2 x^3 - 3 xy)$  are structurally stable vector fields. They are non equivalents and their phase portrait are given in Figure 1. The vector field $X_1$ has a global  unstable  focus at the origin. On the other hand, $X_2$ has four singular points at infinity, two stable and two unstable nodes, therefore, we observe four hyperbolic sectors at origin.
\begin{figure}[h]
\centering
\includegraphics[scale=0.5, bb=26 430 565 699]{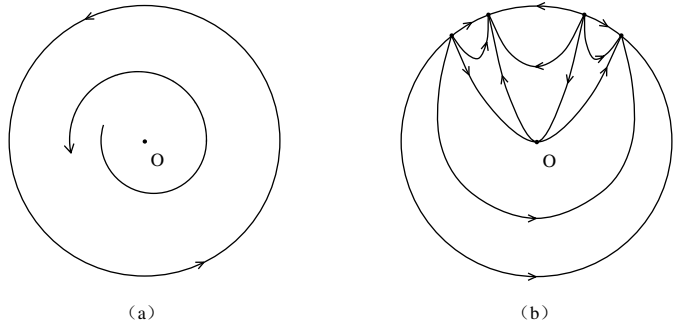}
\caption{Phase portraits of $X_1$ and $X_2$ in the Poincar\'e-Lyapunov disc }
\end{figure}
\end{example}

\section{The number of topological equivalence classes in $\Omega_{pqm}$}\label{s4}

In this section we compute the number of topological equivalence classes in  $\Omega_{pqm}$.

\begin{proof}[Proof of Proposition \ref{p5}] By Lemma \ref{l9}
$\eta(x,y)$ is a $(p,q)$-quasihomogeneous polynomial of degree $p+q+m-1$ having
the form \eqref{eq4}. Then $(i,j)$ in \eqref{eq4} is
a   solution of the equation
\begin{equation}\label{eq29}
 pi+qj=p+q+m-1.
\end{equation}
The pair $(i,j)$ is called an integer (resp. non-negative integer)  solution of  \eqref{eq29}
if $i$ and $j$ are integers (resp. non-negative integers). If
$(i_0,j_0)$ is an integer solution of \eqref{eq29}, then all the
other integer solution are
\begin{equation*}
\{(i_0+lq,\,\,j_0-lp):l\in{\mathbb{Z}}\}.
\end{equation*}

 If $\eta(0,1)\not=0$, $\eta(1,0)\not=0$,
then the equation \eqref{eq29} has two integer solutions $(0,j')$
and $(i',0)$ satisfying $qj'=p+q+m-1$ and $pi'=p+q+m-1$
respectively. Since $(p,q)=1$, we have $pq|(p+q+m-1)$. Hence there
exists $r$ such that $p+q+m-1=(r+1)pq$. Let
$(i_0,j_0)=(0,j')=(0,(p+q+m-1)/q)=(0,(r+1)p)$. Then all non-negative
integer solution of \eqref{eq29} are
$\{(lq,(r+1-l)p):\,\,l=0,1,2,\cdots,r+1\}$, which yields $\eta(x,y)$ has the form, defined in
\eqref{eq5}.  On the other hand, if $\eta(0,1)\not=0$,
$\eta(1,0)\not=0$, then $P(0,1)\not=0,\,\,Q(1,0)\not=0$. This gives that $q|(p+m-1)$ and $p|(q+m-1)$. As
$(p+m-1)/q=(p+q+m-1)/q-1=(r+1)p-1$ and $(q+m-1)/p=(p+q+m-1)/p-1=(r+1)q-1$, we get that $P(x,y)$ and $Q(x,y)$
has the form, defined in
\eqref{eq5}, by the same arguments as above.

Suppose that  $\eta(0,1)\not=0$, $\eta(1,0)=0$. As
$X\in\Omega_{pqm}$, it follows from Theorem \ref{th4} that $y=0$ is a
simple zero of $\eta(1,y)$. By the assumption, the equation
\eqref{eq29} has two integer solutions $(0,j'),\,(i',1)$
respectively and $c_{0,j'}\not=0,\,\,c_{i',1}\not=0$, which implies
that $q|p+m-1$ and $p|m-1$. Let $m-1=k'p$. Then $p+m-1=(k'+1)p$.
Since $(p,q)=1$ we have $q|(k'+1)$. Hence there exists $r$ such
that $p+m-1=rpq$. All non-negative integer solution of \eqref{eq29}
are $\{lq,(r-l)p+1:\,\,l=0,1,2,\cdots, r\}$, which yields that $\eta(x,y)$ has the form in
\eqref{eq6}. On the other hand, if $\eta(0,1)\not=0$,
$\eta(1,0)=0$, then $P(0,1)\not=0,\,\,Q(1,0)=0$. We get $P(x,y)$ and $Q(x,y)$, defined in
\eqref{eq6}, by the same arguments as above.

If $\eta(0,1)=0$ and $\eta(1,0)\not=0$, then \eqref{eq29} has two
integer solutions $(1,j'),\,\,(i',0)$  and $x=0$ is a simple zero of
$\eta(x,1)$. Using the same arguments as above, we get (c).

If $\eta(0,1)=0$ and $\eta(1,0)=0$, then \eqref{eq29} has two
integer solutions $(1,j'),\,\,(i',1)$  and $x=0$ and $y=0$ are
simple zeros of $\eta(x,1)$ and $\eta(1,y)$ respectively. Using the
same arguments as above, one gets (d).
\end{proof}

\begin{proof}[Proof of Lemma \ref{l4.1}]
For the first part of  this lemma, we only prove the case (a.1). Other results are proved by the same arguments.

If  $(p,q,m)\in \Theta_1\cap \Theta_4$, then
\begin{equation*}
p+q+m-1=(r_1+1)pq,\,\,\,m-1=(r_4-1)pq.
\end{equation*}
Eliminating $m$ from the above equations, we get
\begin{equation}\label{eq29.0}
p+q=(r_1-r_4+2)pq,
\end{equation}
which implies $p|q$. Since we suppose $(p,q)=1$, one obtains $p=1$. Substituting $p=1$ into \eqref{eq29.0}, we have
$1=(r_1-r_4+1)q$. Therefore $q=1$ and $r_1=r_4$.

The second part of this lemma  is proved by direct computations.
\end{proof}

\begin{proposition}\label{p14}
Let $p,\,q,\,m,\,r$ be the integers defined in Proposition \ref{p5}.
\begin{itemize}
\item[(a)] Suppose that $p$ and $q$ are odd, then $m$ is odd (resp.
even) if and only if $r$  is odd (resp. even),
\item[(b)] Suppose that $p$ is odd and $q$ is even.
\begin{itemize}
\item[(i)] If $(p,q,m)\in \Theta_1\cup \Theta_2$, then $m$ is even,
\item[(ii)]If  $(p,q,m)\in \Theta_3\cup \Theta_4$, then  $m$ is odd.
\end{itemize}
\end{itemize}
\end{proposition}
\begin{proof}
Suppose $(p,q,m)\in \Theta_1$, then   $p+q+m-1=(r+1)pq$. If $p$ and $q$ are odd, then $p+q-1$ is
odd. As $m=(r+1)pq-(p+q-1)$, we have that $m$ is odd (resp. even)
if and only if $r$  is odd (resp. even).  If $p$ is odd and $q$ is
even, then $p-1$ is even, which implies that $m=(r+1)pq-q-(p-1)$ is
even.

The other statements can be proved by the same arguments.
\end{proof}

\begin{proposition}\label{p15}
 If $X\in\Omega_{pqm}$, then there exists a  number $k$  such that $X$ has $2k$ singular
points at infinity  with  $k\leq r+1$ and $k\equiv r+1$ (mod 2),
where $r$ is defined in  Proposition \ref{p5}.
\end{proposition}
\begin{proof}
Consider the quasihomogeneous polynomial $\eta(x,y)$, defined in  \eqref{eq4}. Since $p$ is odd, $\eta(1,y)$ has at most $r+1$ real
zeros if $\eta(0,1)\not=0$ and $r$ real zeros if $\eta(0,1)=0$. On
the other hand, $x=0$ is  a simple zero of $\eta(x,1)$ if
$\eta(0,1)=0$.  The statement
follows from Lemma \ref{l11}.
\end{proof}

Let
\begin{equation*}
\Omega_{pqm}^{2k}=\{X\in\Omega_{pqm}: E(X) \,\,\mathrm{has}\,\, 2k
\,\,\mathrm{singular\,\, point\,\, at\,\, infinity}\}
\end{equation*}

The following corollary follows from Proposition \ref{p15}:
\begin{corollary}\label{c16}
Suppose that $r$ is as defined in Proposition \ref{p5}, then
$\Omega_{pqm}=\bigcup_{k\in J_{m,r}} \Omega_{pqm}^{2k}$, where
\begin{equation*}
 J_{m,r}=\left\{k=2j:\,\,0\leq j\leq \frac{r+1}{2}\right\}.
 \end{equation*}
if $r$ is odd, and
\begin{equation*}
 J_{m,r}=\left\{k=2j+1:\,\,0\leq j\leq \frac{r}{2}\right\}.
 \end{equation*}
if $r$ is even, respectively.
\end{corollary}

\vskip 0.3cm

 If $X$ and $Y$ are two topological equivalent vector fields in
 $\Omega_{pqm}$, then they have the same number of singular points at
 infinity. Let  $C_{pqm}^k$ be the number of topological
 equivalence  classes in $\Omega_{pqm}^{2k}$. If $k=0$, then  by Theorem \ref{th3}  $C_{pqm}^0=2$ (a global stable
focus and a global unstable focus). It follows from Corollary
 \ref{c16} that
 \begin{equation}\label{cpqm}
  C_{pqm}=\sum_{k\in J_{m,r}}C_{pqm}^k.
 \end{equation}

To convenience, we call $\lambda=+\infty$ (resp. $\lambda=-\infty$)
 a  simple zero of $\eta(1,u)$ (resp. $\eta(-1,u)$) if $\eta(0,1)=0$ (resp.
$\eta(0,-1)=0$) and $\partial \eta(0,1)/\partial x\not=0$. Define
$\mathrm{sgn}(\partial\eta(\pm 1,\pm\infty)/\partial u)= -
\mathrm{sgn}(\partial\eta(0,\pm 1)/\partial x)$ if $\eta(0,\pm
1)=0$. In what follows we suppose that
$\lambda_1,\,\lambda_2,\,\cdots,\lambda_k$ are zeros of $\eta(1,u)$
and $\lambda_{k+1},\,\lambda_{k+2},\,\cdots,\lambda_{2k}$ are zeros
of $\eta(-1,u)$ with $-\infty<\lambda_1<\lambda_2<\cdots<\lambda_k$
and $+\infty>\lambda_{k+1}>\lambda_{k+2}>\cdots>\lambda_{2k}$
respectively, provided $X=(P,Q)\in\Omega_{pqm}^{2k}$.

\begin{proposition}\label{p17}
Let $X=(P,Q)\in\Omega_{pqm}^{2k}$, Then
\begin{itemize}
\item[(a)] $(\partial\eta(1,\lambda_i)/\partial
u)(\partial\eta(1,\lambda_{i+1})/\partial u)<0$,
$i=1,2,\cdots,k-1,\,\,k\geq 2$;
\item[(b)] $(\partial\eta(-1,\lambda_i)/\partial
u)(\partial\eta(-1,\lambda_{i+1})/\partial u)<0$ for
$i=k+1,\cdots,2k-1,\,\,k\geq 2$;
\item[(c)] $P(1,\lambda_i)\not=0,\,\,i=1,2,\cdots,k$, if
$\lambda_k\not=+\infty$, and
$P(-1,\lambda_i)\not=0,\,\,i=k+1,\cdots,2k$, if
$\lambda_{2k}\not=-\infty$;
\item[(d)] $Q(0,\pm 1)\not=0$ if $\eta(0,\pm 1)=0$; and
\item[(e)] $(\partial\eta(1,\lambda_k)/\partial
u)(\partial\eta(-1,\lambda_{k+1})/\partial u)>0$,
$(\partial\eta(1,\lambda_1)/\partial
u)(\partial\eta(-1,\lambda_{2k})/\partial u)>0$.
\end{itemize}
\end{proposition}
\begin{proof}
As $X\in\Omega_{pqm}$, all zeros of
$\eta(\pm 1,u)$ are simple. The statements (a)  and (b) follow if
$\lambda_{k}\not=+\infty$ and  $\lambda_{2k}\not=-\infty$.

Suppose  $\lambda_k=+\infty$, which yields  $\eta(0,1)=0$ and
$\eta(1,u)\not=0$ for $u>\lambda_{k-1}$. Proposition \ref{p5} shows that
$\eta(x,y)$ has the form as \eqref{eq7} or \eqref{eq8}. If
$\partial\eta(1,\lambda_{k-1})/\partial u>0$ , then
$\eta(1,u)>0,\,\,u\in(\lambda_{k-1},+\infty)$. This implies either
$c_{1,rp}>0$ in \eqref{eq7} or $c_{1,(r-1)p+1}>0$   in
\eqref{eq8}. Since either $\partial \eta(0,1)/\partial
x=c_{1,rp}$ or $\partial \eta(0,1)/\partial x=c_{1,(r-1)p+1}$ happens
and $\mathrm{sgn}(\partial\eta(\pm 1,\pm\infty)/\partial u)= -
\mathrm{sgn}(\partial\eta(0,\pm 1)/\partial x)$, one gets (a) for
$i=k$ if $\partial\eta(1,\lambda_{k-1})/\partial u>0$.  By the same
arguments we get    $\partial\eta(1,\lambda_k)/\partial u>0$ if $\lambda_k=+\infty$,
$\partial\eta(1,\lambda_{k-1})/\partial u<0$.
This proves  (a).

If $\lambda_{2k}=-\infty$,  then $\eta(-1,u)\not=0$ for
$u\in(-\infty,\lambda_{2k-1})$. Repeating the same arguments as the
proof of (a), one gets (b).

The statement (c) and (d) have been proved in Section \ref{s2}.

Finally we prove (e).  Suppose that  $\eta(x,y)$
has the form   \eqref{eq5}. From the
definition of $\lambda_k$ and $\lambda_{k+1}$  we have
$\eta(1,u)\not=0$ in the interval $(\lambda_k,+\infty)$ and
$\eta(-1,u)\not=0$ in the interval $(\lambda_{k+1},+\infty)$
respectively. If $\eta(x,y)$ has the form \eqref{eq5}, then
\begin{equation*}
\eta(\pm 1, u)=c_{0,(r+1)p}u^{(r+1)p}+(\pm 1)c_{1,rp}u^{rp}+\cdots+c_{rq,p}(\pm 1)^{rq}u^p+c_{(r+1)q,0}(\pm 1)^{(r+1)q},
\end{equation*}
which gives $\mathrm{sgn}\eta(1,u)=\mathrm{sgn}\eta(-1,u)$ as
$u\rightarrow +\infty$.  As $\lambda_k$ and $\lambda_{k+1}$ are
simple zeros of  $\eta(1,u)$ and $\eta(-1,u)$ respectively,   one
gets \linebreak $(\partial\eta(1,\lambda_k)/\partial
u)(\partial\eta(-1,\lambda_{k+1})/\partial u)>0$. On the other hand,
$\mathrm{sgn}\eta(1,u)=\mathrm{sgn}\eta(-1,u)$ as $u\rightarrow
-\infty$, which implies $(\partial\eta(1,\lambda_1)/\partial
u)(\partial\eta(-1,\lambda_{2k})/\partial u)>0$.

Suppose that $\eta(x,y)$ has the form
\eqref{eq7}, then
$\lambda_k=+\infty,\,\,\lambda_{2k}=-\infty$, and
\begin{equation*}
\eta(x,\pm 1)=x(c_{1,rp}(\pm 1)^{rp}+c_{1+q,(r-1)p}(\pm
1)^{(r-1)p}x^q+\cdots+c_{1+rq,0}x^{rq}),
\end{equation*}
which yields $\mathrm{sgn}(\partial\eta(\pm 1,\pm\infty)/\partial
u)=-\mathrm{sgn}(\partial\eta(0,\pm 1)/\partial x)=-\mathrm{sgn}
(c_{1,rp}(\pm 1)^{rp})$. As
\begin{equation*}
\eta(\pm 1,u)=(\pm 1)(c_{1,rp}u^{rp}+c_{1+q,(r-1)p}(\pm
1)^qu^{(r-1)p}+\cdots+c_{1+rq,0}(\pm 1)^{rq}),
\end{equation*}
$\mathrm{sgn}(\eta(-1,u))=-\mathrm{sgn}(c_{1,rp})$ for
$u\in(\lambda_{k+1},+\infty)$  and
$\mathrm{sgn}(\eta(1,u))=\mathrm{sgn}(c_{1,rp}(-1)^{rp})$ for
$u\in(-\infty,\lambda_1)$  respectively,  which shows that $\mathrm{sgn}(\partial
\eta(-1,\lambda_{k+1})/\partial u)=-\mathrm{sgn}(c_{1,rp})$ and
$\mathrm{sgn}(\partial \eta(1,\lambda_1)/\partial
u)=-\mathrm{sgn}(c_{1,rp}(-1)^{rp})$. The statement  (e)
 follows.

If $\eta(x,y)$ is defined as \eqref{eq6} or \eqref{eq8},then one gets (e)  by
the same arguments.
\end{proof}

\medskip

Let $S$ be the set of all sequence
$(\sigma,\nu)=\{(\sigma_i,\nu_i)\}_{i\in\mathbb{Z}}$ such that
$\sigma_i,\nu_i\in\{-1,1\}$ and $\sigma_i\sigma_{i+1}<0$ for all
$i\in\mathbb{Z}$. For each $k\in\mathbb{N}$ we denote
\begin{equation*}
S^{2k}=\{(\sigma,\nu)\in
S:\,(\sigma_i,\nu_i)=(\sigma_{i+2k},\nu_{i+2k}),\,\,\mathrm{for\,\,all}\,\,i\in\mathbb{Z}\}.
\end{equation*}

A sequence $(\sigma,\nu)$ is {\it periodic of period $l$} (or {\it
$l$-periodic}) if $l$ is the smallest natural number such that
$(\sigma_i,\nu_i)=(\sigma_{i+l},\nu_{i+l})$ for all
$i\in\mathbb{Z}$. Obviously each  sequence $(\sigma,\nu)$ in
$S^{2k}$ is  periodic and is completely determined if the  elements
$(\sigma_i,\nu_i)$ are given for $i=1,2,\cdots,l$. As
$\sigma_i\sigma_{i+1}=-1$, the period $l$ is an even divisor of
$2k$.

The above notations  have been used in the study of the homogeneous
vector fields in \cite{LRR}.

From Proposition \ref{p17} we can associate a sequence of $S^{2k}$
to each vector field $X\in\Omega_{pqm}^{2k}$  by taking
\begin{equation}\label{eq30}
(\sigma_i,\nu_i)=\left\{\begin{array}{ll}\left(\mathrm{sgn}\left(\frac{\partial\eta(1,\lambda_i)}{\partial
u}\right),\mathrm{sgn}(-P(1,\lambda_i))\right),&i=1,2,\cdots,k-1;\\[2ex]
\left(\mathrm{sgn}\left(\frac{\partial\eta(1,\lambda_i)}{\partial
u}\right),\mathrm{sgn}(-P(1,\lambda_i))\right),&\mathrm{if}\,\,i=k,\,\,\lambda_k\not=+\infty;\\[2ex]
\left(\mathrm{sgn}\left(\frac{-\partial\eta(0,1)}{\partial
v}\right),\mathrm{sgn}(-Q(0,1))\right),&\mathrm{if}\,\,i=k,\,\,\lambda_k=+\infty;\\[2ex]
\left(\mathrm{sgn}\left(\frac{-\partial\eta(-1,\lambda_i)}{\partial
u}\right),\mathrm{sgn}(P(-1,\lambda_i))\right),&i=k+1,\cdots,2k-1;\\[2ex]
\left(\mathrm{sgn}\left(\frac{-\partial\eta(-1,\lambda_i)}{\partial
u}\right),\mathrm{sgn}(P(-1,\lambda_i))\right),&\mathrm{if}\,\,i=2k,\,\,\lambda_{2k}\not=-\infty;\\[2ex]
\left(\mathrm{sgn}\left(\frac{\partial\eta(0,-1)}{\partial
v}\right),\mathrm{sgn}(Q(0,-1))\right),&\mathrm{if}\,\,i=2k,\,\,\lambda_{2k}=-\infty,
\end{array}
\right.
\end{equation}
where $\eta(v,u)$ is defined by \eqref{eq2}.

\begin{proposition}\label{p18}
Let $X=(P,Q)\in\Omega_{pqm}^{2k}$ with $0\not=k\in J_{m,r}$ and let
$(\sigma,\nu)$ be the sequence associated to $X$ according to
\eqref{eq30}.
\begin{itemize}
\item[(a)] Suppose that  both $p$ and $q$ are odd, then
$(\sigma_{i+k},\nu_{i+k})=(-1)^{m-1}(\sigma_i,\nu_i)$ for
$i=1,2,\cdots,k$;
\item[(b)] Suppose that   $p$ is odd and  $q$ is even.
\begin{itemize}
\item[(i)] If $\eta(0,1)\not=0$,
 then $m$ is even and $(\sigma_{2k-i+1},\nu_{2k-i+1})=-(\sigma_i,\nu_i)$  for
$i=1,2,\cdots,k$.
\item[(ii)] If $\eta(0,1)=0$, then $m$ is odd,  $(\sigma_{2k-i},\nu_{2k-i})=(\sigma_i,\nu_i)$ for
$i=1,2,\cdots,k-1$, and
$(\sigma_{2k},\nu_{2k})=(-1)^{r-1}(\sigma_k,\nu_k)$, where $r$ is
defined in  (c) or (d) of Proposition \ref{p5}.
\end{itemize}
\end{itemize}
\end{proposition}
\begin{proof}
It follows from Lemma \ref{l11} (or \eqref{eq27}) that if $\eta
(1,\lambda_i)=0,\,\,i=1,2,\cdots,k$ and $\lambda_k\not=+\infty$,
then $\eta(-1,(-1)^q\lambda_i)=0$.

(a)\quad Suppose that  both $p$ and $q$ are  odd. As
$\lambda_1<\lambda_2<\cdots<\lambda_k$, we have
$(-1)^q\lambda_1>(-1)^q\lambda_2>\cdots>(-1)^q\lambda_k$, which
implies $\lambda_{i+k}=(-1)^q\lambda_i,\,\,i=1,2,\cdots,k$. As
$P(-1,\lambda_{i+k})=P((-1)^p,(-1)^q\lambda_i)=(-1)^mP(1,\lambda_i)$,
we get $(\sigma_{i+k},\nu_{i+k})=(-1)^{m-1}(\sigma_i,\nu_i)$ from
\eqref{eq28} for $i=1,2,\cdots,k-1$, and for $i=k$ if
$\lambda_k\not=+\infty$ .

If $\lambda_k=+\infty$, then  it follows from Lemma \ref{l11}
$\eta(0,1)=\eta(0,-1)=0$, which implies that $\eta(x,y)$ has the
form  as  \eqref{eq7} or \eqref{eq8}. If $\eta(x,y)$ is defined in
 \eqref{eq7}, then $\partial \eta(0,\pm 1)/\partial
x=c_{1,rp}(\pm 1)^{rp}=c_{1,rp}(\pm 1)^r$ and $q+m-1=rpq$. By
Proposition \ref{p14} $\partial \eta(0,\pm 1)/\partial
x=c_{1,rp}(\pm 1)^m$. As
$Q(0,-1)=(-1)^{q-1+m}Q(0,1)=(-1)^mQ(0,1)$, one gets
$(\sigma_{2k},\nu_{2k})=(-1)^{m-1}(\sigma_k,\nu_k)$.

If $\eta(x,y)$ is defined in  \eqref{eq8}, then  we get
$(\sigma_{2k},\nu_{2k})=(-1)^{m-1}(\sigma_k,\nu_k)$ by the same
arguments.

(b) \quad Suppose that  $p$ is odd and $q$ is  even. If
$\eta(0,1)\not=0$, then $\lambda_{2k-i+1}=\lambda_i$. By the same
arguments we get
$(\sigma_{2k-i+1},\nu_{2k-i+1})=(-1)^{m-1}(\sigma_i,\nu_i)$ for
$i=1,2,\cdots,k$.  On the other hand, if $\eta(0,1)\not=0$, then  it follows  from Proposition \ref{p5} and
 Proposition \ref{p14}(b)(i) that $m$ is even. This proves (i).

If $\eta(0,1)=0$ and $q$ is even, then it follows from  Proposition \ref{p5} and Proposition \ref{p14}(b)(ii)
that $m$ is odd and $\lambda_k=-\lambda_{2k}=+\infty$. By the same
arguments as (a), we have $\lambda_{2k-i}=\lambda_i$ and hence
$(\sigma_{2k-i},\nu_{2k-i})=(\sigma_i,\nu_i)$ for
$i=1,2,\cdots,k-1$.

If $\eta(x,y)$ is defined in  \eqref{eq7}, then it follows from
\eqref{eq7}  that $\partial\eta(0,\pm 1)/\partial
v=c_{1,rp}(-1)^r,\,\,Q(0,\pm 1)=b_{0,rp}(\pm 1)^r$, which gives
$(\sigma_{2k},\nu_{2k})=(-1)^{r-1}(\sigma_k,\nu_k)$. If $\eta(x,y)$
has the form \eqref{eq8}, then we get (ii) by the same arguments.
\end{proof}

Following the  idea in \cite{LRR}, we say that $(\sigma,\nu)$ is
{\it $m$-admissible} if there exists $X\in\Omega_{pqm}^{2k}$ satisfying
\eqref{eq30} for given pair $(p,q)$. Denote by $S_m^{2k}$ the set of all sequence in
$S^{2k}$ that are $m$-admissible.

The next two propositions characterize the sequences in $S^{2k}$
that are $m$-admissible. It follows from Proposition \ref{p5} that, if
$X\in \Omega_{pqm}$, then it  is only necessary to consider the numbers $p,\,q,\,m$ and $r$ that
satisfy one of the equations in  \eqref{eq9}.

\begin{proposition}\label{p19}
Suppose that   $r$ is defined as in Proposition \ref{p5}. Denote by $s$
the number of changes of sign in the sequence
$\{\nu_1,\nu_2,\cdots,\nu_k\}$.  Let $0\not=k\in J_{m,r}$ and
$(\sigma,\nu)=\{(\sigma_i,\nu_i)\}_{i\in\mathbb{Z}}$ verifying
\begin{itemize}
\item[(i)] $(\sigma_{i+k},\nu_{i+k})=(-1)^{m-1}(\sigma_i,\nu_i)$  if
both $p$ and $q$ are odd,
\item[(ii)] $(\sigma_{2k-i+1},\nu_{2k-i+1})=-(\sigma_i,\nu_i)$ if
$p$ is odd and  $q,\,m$ are even,
\item[(iii)] $(\sigma_{2k-i},\nu_{2k-i})=(\sigma_i,\nu_i)$ for $i=1,2,\cdots,k-1$
and $(\sigma_{2k},\nu_{2k})=(-1)^{r-1}(\sigma_k,\upsilon_k)$  if
$p,\,m$ are  odd and $q$ is  even.
\end{itemize}
Then
\begin{itemize}
\item[(a)] If either $k<r+1$ or $k=r+1,\,\,s<r$, then $(\sigma,\nu)$ is $m$-admissible.
\item[(b)] If $k=r+1,\,s=r$, then $(\sigma,\nu)$ is $m$-admissible if and
only if there exists $j\in\mathbb{Z}$ such that $\sigma_j=\nu_j$.
\end{itemize}
\end{proposition}
\begin{proof}
(a)  To prove the statement  (a) it is sufficient to find
$X=(P(x,y),Q(x,y))$ such that $X\in\Omega_{pqm}^{2k}$ and \eqref{eq30}
is fulfilled  for $i=1,2,\cdots,k$.

We note that $k<r+1$ implies $s<r$.

{\it Case 1.  $(p,q,m)\in\Theta_1$}.

  Let $\lambda_1<\lambda_2<\cdots<\lambda_k$ be the
finite real numbers with $\lambda_j\not=0,\,\,j=1,2,\cdots,k$ and
\begin{equation*}
\eta(x,y)=a(x^{2q}+y^{2p})^{(r+1-k)/2}\prod_{j=1}^k(y^p-\lambda_j^px^q)
\end{equation*}
with $a\in\mathbb{R},\,a\not=0$.  Since $k\in J_{m,r}$, it follows
from Proposition \ref{p15} that $(r+1-k)/2\in \mathbb{N}$, and hence
$\eta(x,y)$ is a $(p,q)$ quasihomogeneous polynomial of weight
degree $p+q+m-1=(r+1)pq$, which implies that $\eta(1,u)$ has  simple
real zeros $\lambda_1,\,\lambda_2,\,\cdots,\lambda_k$. Furthermore,
\begin{equation*}
\frac{\partial \eta(1,\lambda_i)}{\partial
u}=ap\lambda_i^{p-1}(1+\lambda_i^{2p})^{(r+1-k)/2}\prod_{j=1,j\not=i}^k(\lambda_i^p-\lambda_j^p)\not=0
,\,\,\,\,\,\,i=0,1,\cdots,k.
\end{equation*}
If we choose $a=\pm 1$ in such way that
$\mathrm{sgn}\partial\eta(1,\lambda_1)/\partial u=\sigma_1$, then
$\eta(x,y)$ satisfies
$\mathrm{sgn}\partial\eta(1,\lambda_i)/\partial
u=\sigma_i,\,\,i=1,2,\cdots,k$.

To determine $P(x,y)$ we choose $\mu_i\in(\lambda_j,\lambda_{j+1})$
if $\nu_j\not=\nu_{j+1}$. So we obtain $s$ real numbers
$\mu_1<\mu_2<\cdots<\mu_s$, where $s$ is the number of changes of
sign in the sequence $\{\nu_1,\nu_2,\cdots,\nu_k\}$. Let
\begin{equation*}
h(x,y)=\prod_{i=1}^s(y^p-\mu_i^px^q).
\end{equation*}
If  $\mathrm{sgn}(ah(1,\lambda_1))=-\nu_1$ , then define
\begin{equation*}
\alpha(x,y)=\left\{\begin{array}{cl}
(x^{2q}+y^{2p})^{(r-s-1)/2}(y^p+(1-\lambda_1^p)x^q),&\mathrm{if}\,\, r-s\,\,\mathrm{ is \,\,odd},\\
(x^{2q}+y^{2p})^{(r-s)/2}, &\mathrm{if}\,\, r-s\,\,\mathrm{ is
\,\,even},
\end{array}
 \right.
\end{equation*}
 If  $\mathrm{sgn}(ah(1,\lambda_1))=\nu_1$ , then
\begin{equation*}
\alpha(x,y)=\left\{\begin{array}{cl}
(x^{2q}+y^{2p})^{(r-s-1)/2}(y^p-(1+\lambda_k^p)x^q),&\mathrm{if}\,\, r-s\,\,\mathrm{ is \,\,odd},\\
(x^{2q}+y^{2p})^{(r-s-2)/2}(y^p+(1-\lambda_1^p)x^q)(y^p-(1+\lambda_k^p)x^q),
&\mathrm{if}\,\, r-s\,\,\mathrm{ is \,\,even}.
\end{array}
 \right.
\end{equation*}
Define $P(x,y)$ and $Q(x,y)$ as follows
\begin{equation*}
P(x,y)=-\frac{a}{q}\alpha(x,y)h(x,y)y^{p-1},\,\,Q(x,y)=\frac{1}{px}(\eta(x,y)+qyP(x,y)).
\end{equation*}
As $\eta(0,y)+qyP(0,y)\equiv 0$, $Q(x,y)$ is also a polynomial.
It is easy to prove that $X=(P,Q)$ is a $(p,q)$-quasihomogeneous
vector filed of degree $m$. From Theorem \ref{th4} and
Proposition \ref{p18}  it follows that $X\in \Omega_{pqm}^{2k}$ and
$(\sigma,\nu)$ associated to $X$ verifies either (i) or (ii).
Therefore $(\sigma,\nu)$ is $m$-admissible.

{\it Case 2.  $(p,q,m)\in\Theta_2$.}

Let $\lambda_1<\lambda_2<\cdots<\lambda_k$ be the finite real numbers
with $0\in\{\lambda_1,\lambda_2,\cdots,\lambda_k\}$. Without loss of
generality suppose $\lambda_l=0$. Let
\begin{eqnarray*}
\eta(x,y)&=&a(x^{2q}+y^{2p})^{(r+1-k)/2}y\prod_{i=1,i\not=l}^k(y^p-\lambda_i^px^q),\\
P(x,y)&=&-\frac{a}{q}\alpha(x,y)h(x,y),
\end{eqnarray*}
 where $a=\pm 1,\,\alpha(x,y),\,\,h(x,y)$ are
defined in the same way as in Case 1. Define $Q(x,y)$ by the
equation \eqref{eq2}. Using the same arguments as
in Case 1, $X=(P,Q)\in \Omega_{pqm}^{2k}$ and $(\sigma,\nu)$ associated
to $X$ verifies either (i) or (ii). Therefore $(\sigma,\nu)$ is
$m$-admissible.

{\it Case 3. $(p,q,m)\in\Theta_3$}.

Let $\lambda_1<\lambda_2<\cdots<\lambda_{k-1}$ be the finite real
numbers with $\lambda_i\not=0,\,\,i=1,2,\cdots,k-1$,
$\lambda_k=+\infty$,  and
\begin{equation*}
\eta(x,y)=ax(x^{2q}+y^{2p})^{(r+1-k)/2}\prod_{i=1}^{k-1}(y^p-\lambda_i^px^q),
\end{equation*}
where $a$ is chose in the same way as in Case 1. If  $\mathrm{sgn}(ah(1,\lambda_1))=-\nu_1$ , then define
\begin{equation*}
\alpha(x,y)=\left\{\begin{array}{cl} (x^{2q}+y^{2p})^{(r-s-1)/2},
&\mathrm{if}\,\, r-s\,\,\mathrm{ is \,\,odd},\\
(x^{2q}+y^{2p})^{(r-s-2)/2}(y^p+(1-\lambda_1^p)x^q),&\mathrm{if}\,\,
r-s\,\,\mathrm{ is \,\,even} .
\end{array}
 \right.
\end{equation*}
 If
$\mathrm{sgn}(ah(1,\lambda_1))=\nu_1$, then
\begin{equation*}
\alpha(x,y)=\left\{\begin{array}{cl}
-x^q(x^{2q}+y^{2p})^{(r-s-3)/2}(y^p+(1-\lambda_1^p)x^q),
&\mathrm{if}\,\, r-s\,\,\mathrm{ is \,\,odd},\\
-x^q(x^{2q}+y^{2p})^{(r-s-2)/2},&\mathrm{if}\,\, r-s\,\,\mathrm{ is
\,\,even}.
\end{array}
 \right.
\end{equation*}
Let
\begin{equation*}
P(x,y)=-\frac{a}{q}xy^{p-1}\alpha(x,y)h(x,y),\,\,\,\,Q(x,y)=\frac{1}{px}(\eta(x,y)+qyP(x,y)),
\end{equation*}
 where $h(x,y)$ is defined as in Case 1. Using the same arguments as
in Case 1, $X=(P,Q)\in \Omega_{pqm}^{2k}$ and $(\sigma,\nu)$ associated
to $X$ verifies either (i) or (iii). Therefore $(\sigma,\nu)$ is
$m$-admissible.

{\it Case 4. $(p,q,m)\in\Theta_4$.}

 Let
$\lambda_1<\lambda_2<\cdots<\lambda_{k-1}$ be the finite real numbers
with $0\in\{\lambda_1,\lambda_2,\cdots,\lambda_{k-1}\}$, and
$\lambda_k=+\infty$. Without loss of generality suppose
$\lambda_l=0$. Let
\begin{eqnarray*}
\eta(x,y)=axy(x^{2q}+y^{2p})^{(r+1-k)/2}\prod_{i=1,i\not=l}^{k-1}(y^p-\lambda_i^px^q),\,\,
P(x,y)=-\frac{a}{q}x\alpha(x,y)h(x,y),
\end{eqnarray*}
where $h(x,y)$ and $\alpha(x,y)$ are  defined as in Case 1 and Case 3 respectively, $a=\pm 1 $ is chose such that
  $\mathrm{sgn}\partial\eta(1,\lambda_1)/\partial u=\sigma_1$.
 Define $Q(x,y)$ by the
equation \eqref{eq2}. We get a  a
$(p,q)$-quasihomogeneous vector filed $X=(P,Q)$ of degree
$m$ which is structurally
stable.  The sequence  $(\sigma,\nu)$ associated to $X$ verifies either (i) or
(iii). Therefore $(\sigma,\nu)$ is $m$-admissible.

(b) Let $k=r+1,\,\,s=r$. Suppose that $(\sigma,\nu)$ is
$m$-admissible, then there exists $X=(P,Q)\in\Omega_{pqm}^{2k}$ such
that $(\sigma,\nu)$ associated to $X$ verifies one of  (i), (ii) and
(iii) and  $\eta(x,y),\,P(x,y),\,Q(x,y)$ have one of the forms
listed in Proposition \ref{p5}.

If $\eta(0,1)\not=0,\,\eta(1,0)\not=0$, then $\eta(x,y)$  has the
form \eqref{eq5}, which implies that
$\eta(1,u)=a\prod_{i=1}^{r+1}(u^p-\lambda_i^p)$, where
$\lambda_1<\lambda_2<\cdots<\lambda_{r+1}$ with
$\lambda_i\not=0,\,\,i=1,2,\cdots,r+1$ and
$a\in\mathbb{R},\,\,a\not=0$. This gives
\begin{equation*}
\eta(x,y)=(x^{1/p})^{p+q+m-1}\eta\left(1,\frac{y}{x^{q/p}}\right)=a\prod_{i=1}^{r+1}(y^p-\lambda_i^px^q).
\end{equation*}
Since $s=r$, $P(1,y)$ has $r$ real zeros which belongs to the interval $(\lambda_1,\lambda_{k+1})$. It follows from
\eqref{eq5} that $P(x,y)$ is necessarily of the form
\begin{equation*}
P(x,y)=b\left(\prod_{i=1}^r(y^p-\mu_i^px^q)\right)y^{p-1},\,\,\,\lambda_1<\mu_1<\lambda_2<\cdots<\lambda_r<\mu_r<\lambda_{r+1},
\end{equation*}
with $b\in\mathbb{R},\,\,b\not=0$.
$Q(x,y)$ is defined by the equation \eqref{eq2}.
Hence  we have $\eta(0,1)=a=-qb=-qP(0,1)$. Therefore
$\sigma_1=\mathrm{sgn}(\partial\eta(1,\lambda_1)/\partial
u)=\mathrm{sgn}((-1)^ra)=\mathrm{sgn}((-1)^{r+1}b)=-\mathrm{sgn}(P(1,\lambda_1))=\nu_1$.

If $\eta(0,1)=0,\,\eta(1,0)\not=0$, then $\eta(1,u)$ has $r+1$ zeros
$\lambda_1<\lambda_2<\cdots<\lambda_r<\lambda_{r+1}=+\infty$ with
$\lambda_i\not=0,\,i=1,2,\cdots,r$. By the same arguments as above,
we have
\begin{equation*}
\eta(x,y)=ax\prod_{i=1}^{r}(y^p-\lambda_i^px^q),\,\,P(x,y)=bx\left(\prod_{i=1}^{r-1}(y^p-\mu_i^px^q)\right)y^{p-1}
\end{equation*}
with $ab\not=0,\,a,b\in\mathbb{R}$, and $\lambda_1<\mu_1<\lambda_2<\cdots<\mu_{r-1}<\lambda_r<+\infty$. This gives
\begin{equation*}
Q(x,y)=\frac{1}{p}\left(a\prod_{i=1}^{r}(y^p-\lambda_i^px^q)+qb\left(\prod_{i=1}^{r-1}(y^p-\mu_i^px^q)\right)y^p\right).
\end{equation*}
Hence
$\nu_r=\mathrm{sgn}(-P(1,\lambda_r))=\mathrm{sgn}(-b),\,\,\nu_{r+1}=\mathrm{sgn}(-Q(0,1))=\mathrm{sgn}(-a-qb)$.
Since we suppose $s=r$, we have $\nu_r\nu_{r+1}=-1$, which implies
$ab<0$. This yields
$\sigma_1=\mathrm{sgn}(\partial\eta(1,\lambda_1)/\partial
u)=\mathrm{sgn}(a(-1)^{r-1})$,
$\nu_1=\mathrm{sgn}(-P(1,\lambda_1))=\mathrm{sgn}(-b(-1)^{r-1})$.
Finally, one gets  $\sigma_1=\nu_1$.

If either $\eta(0,1)\not=0,\,\eta(1,0)=0$ or
$\eta(0,1)=\eta(1,0)=0$, one gets $\sigma_1=\nu_1$ by the same
arguments.

Conversely, if $\sigma_1=\nu_1$ and $k=r+1,\,s=r$, then there exist
the vector field $X=(P,Q)$ such that $(\sigma,\nu)$ verifies one of
(i), (ii) and (iii), where $P(x,y),\,Q(x,y)$ are defined as above.
This finishes proof.
\end{proof}

\begin{proposition}\label{p20}
Let $0\not=k\in J_{m,r}$ and $X,\,X'\in\Omega_{pqm}^{2k}$.  Denote by
$(\sigma,\nu),\,(\sigma',\nu')\in S_m^{2k}$ the sequences associated
to $X$ and $X'$ respectively. Then $X$ and $X'$ are topologically
equivalent if and only if there exists $\tau\in \mathbb{N}$ ($0\leq
\tau\leq 2k-1$) such that one of the following conditions is satisfied for
all $i\in \mathbb{Z}$:
\begin{equation*}
(\sigma_i,\nu_i)=(\sigma'_{i+\tau},\nu_{i+\tau}'),
\end{equation*}
\begin{equation*}
(\sigma_i,\nu_i)=(\sigma'_{2k-i+1+\tau},\nu_{2k-i+1+\tau}').
\end{equation*}
\end{proposition}
\begin{proof}
The statement is proved by following the arguments in  the proof of
Proposition  11 of \cite{LRR}.
\end{proof}

Let  $w=(\sigma,\nu)=\{w_i=(\sigma_i,\nu_i),\,i\in\mathbb{Z}\}\in
S_{m}^{2k}$. It follows from Proposition \ref{p19} that the sequence
$\bar w=\{{\bar w}_i=(\sigma_{i+1},\nu_{i+1}),\,i\in\mathbb{Z}\}$
also belongs $S_{m}^{2k}$. Define the application
\begin{equation*}
\mathfrak{R}:\,\,S_m^{2k}\rightarrow S_m^{2k}
\end{equation*}
such that if $w=(\sigma,\nu)\in S_m^{2k}$, then $\mathfrak{R}(w)$ is
the sequence of $S_m^{2k}$ verifying
\begin{equation}\label{eq31}
(\mathfrak{R}(w))_i=w_{i+1}\,\,\,\,\,\,\mathrm{for
\,\,all}\,\,i\in\mathbb{Z}.
\end{equation}
Whenever $f:\,\,A\rightarrow A$ is an arbitrary application then
$a\in A$ is called a {\it $l$-periodic point} of $f$ if $f^l(a)=a$
and $f^i(a)\not=a$ for $i=1,2,\cdots,l-1$. The integer $l$ is called
the {\it period} of $a$.  The set
$C=\{a,f(a),\cdots,f^{l-1}(a)\}\subset A$ is a {\it cycle of order
$l$} ({\it $l$-cycle}) of $f$ if $a\in A$ is a $l$-periodic point of
$f$.  These notation can be found in many papers, see for instance
\cite{LRR}.

From \eqref{eq31} $w\in S_m^{2k}$ is a $l$-periodic point of
$\mathfrak{R}$ if and only if $w$ is a $l$-periodic sequence.
Therefore $C\subset S_m^{2k}$ is a $l$-cycle of $\mathfrak{R}$ if
there exists a $l$-periodic sequence $w\in S_m^{2k}$ such that
$C=\{w,\mathfrak{R}(w),\cdots,\mathfrak{R}^{l-1}(w)\}$. Each
sequence $w\in S_m^{2k}$ belongs to some cycle of order $l$ of
$\mathfrak{R}$ where $l$ is an even division of $2k$\cite{LRR}.

Define the application
\begin{equation*}
\Psi:\,\,\,S_m^{2k}\rightarrow S_m^{2k},
\end{equation*}
where if $w\in S_m^{2k}$ then $\Psi(w)$ is given by
\begin{equation*}
(\Psi(w))_i=w_{2k-i+1}.
\end{equation*}
By the definition we know that $\Psi$ is the application that
reverses the order of the elements
$(\sigma_1,\nu_1),\,(\sigma_2,\nu_2),\cdots,(\sigma_{2k},\nu_{2k})$.

The following two propositions have been obtained  in \cite{LRR} for
homogeneous polynomial vector fields. They  can be proved  in the
same way for $(p,q)$-quasi-homogenous polynomial vector fields. Here
we omit the details.

\begin{proposition}\label{p21}
Let $0\not=k\in J_{m,r}$. The following statements are true:
\begin{itemize}
\item[(a)] $\Psi\circ\Psi=Id$;
\item[(b)] $\Psi\circ\mathfrak{R}=\mathfrak{R}^{-1}\circ\Psi$.
\end{itemize}
\end{proposition}

\begin{proposition}\label{p22}
Let $0\not=k\in J_{m,r}$ and $X,\,X'\in\Omega_{pqm}^{2k}$. Denote  by
$w$ and $w'$ the sequences of $S_m^{2k}$ associated $X$ and $X'$
respectively. Then $X$ and $X'$ are topologically equivalent if and
only if  either $w$ and $w'$, or $w$ and $\Psi(w')$ belong to the
same cycle of $\mathfrak{R}$.
\end{proposition}

Following the idea in \cite{LRR},  we define in $S_m^{2k}$ the
equivalent relation: $w\thicksim \bar w$ if and only if one of the
sequences $w$ and $\Psi(w)$ belongs to the cycle $\{\bar w,
\mathfrak{R}(\bar w),\cdots, \mathfrak{R}^{l-1}(\bar w)\}$.
According to Proposition \ref{p22},  the number of equivalence classes of
$\thicksim$ in $S_m^{2k}$ coincide with the number $C_{pqm}^k$ of
topological equivalence classes in $\Omega_{pqm}^{2k}$.

Consider the cycle $C$ of $\mathfrak{R}$ such that $w\in C$. From
Proposition \ref{p21}(b),
$\Psi(\mathfrak{R}^i(w))=\mathfrak{R}^{-i}(\Psi(w))$, and hence the
cycle of $\Psi(w)$ is $C'=\{\Psi(w):w\in C\}$. Therefore the
equivalence class of $w$ is $C\cup C'$.

A cycle $C$ of $\mathfrak{R}$ is called {\it symmetric} if $C=C'$.
Denote by $D_{pqm}^k$ the number of cycles of $\mathfrak{R}$ in
$S_m^{2k}$ and $E_{pqm}^k$ the number of symmetric cycles respectively.
Then
\begin{equation}\label{eq32}
C_{pqm}^k=E_{pqm}^k+\frac{D_{pqm}^k-E_{pqm}^k}{2}=\frac{E_{pqm}^k+D_{pqm}^k}{2}.
\end{equation}
The above expression has been obtained in \cite{LRR} for homogeneous
polynomial systems.

\begin{proposition}\label{p23}
Suppose that $p$ and $q$ are odd and $0\not=k\in J_{m,r}$.
\begin{itemize}
\item[(a)] If $r$ is odd and $k<r+1$, then
\begin{equation*}
D_{pqm}^k=\sum_{2n|k}\mathcal{
P}_{2n},\,\,\mathcal{P}_{2n}=\frac{1}{n}\left(2^{2n}-\sum_{l|n,l\not=n}l\mathcal{
P}_{2l}\right).
\end{equation*}
\item[(b)]If $r$ is even and $k<r+1$, then
\begin{equation*}
D_{pqm}^k=\sum_{n|k}\mathcal{
P}_{2n},\,\,\mathcal{P}_{2n}=\frac{1}{n}\left(2^n-\sum_{l|n,l\not=n}l\mathcal{
P}_{2l}\right).
\end{equation*}
\item[(c)] If $r$ is odd, then
$D_{pqm}^{r+1}=(\sum_{2n|r+1}\mathcal{P}_{2n})-1$  with
$\mathcal{P}_{2n}$ defined as in (a).
\item[(d)] If $r$ is even, then
$D_{pqm}^{r+1}=(\sum_{n|r+1}\mathcal{P}_{2n})-1$  with
$\mathcal{P}_{2n}$ defined as in (b).
\end{itemize}
\end{proposition}
\begin{proof}
It follows from Proposition \ref{p14} that $m$ is odd (resp. even)
if and only if $r$ is odd (resp. even). The proposition follows by
the same arguments as in \cite{LRR}.
\end{proof}

The sequences verifying Proposition \ref{p18}(b) are not occurred
in \cite{LRR}. We study $D_{pqm}^k$ for these sequences in the following
proposition.
\begin{proposition}\label{p24}
Suppose that $p$ is odd and  $q$ is  even,  $0\not=k\in J_{m,r}$.
\begin{itemize}
\item[(a)] If $(p,q,m)\in\Theta_1\cup\Theta_2$,  then
\begin{equation*}
D_{pqm}^k=\sum_{n|k}\mathcal{
P}_{2n},\,\,\mathcal{P}_{2n}=\frac{1}{n}\left(2^n-\sum_{l|n,l\not=n}l\mathcal{
P}_{2l}\right),
\end{equation*}
for $k<r+1$, and $D_{pqm}^{r+1}=(\sum_{n|r+1}\mathcal{P}_{2n})-1$ for
$k=r+1$, respectively.
\item[(b)] If either $(p,q,m)\in\Theta_3\cup\Theta_4$, then $D_{pqm}^k=\sum_{n|k}\mathcal{
P}_{2n}$ for $0\not=k<r+1$ and $D_{pqm}^k=(\sum_{n|r+1}\mathcal{P}_{2n})-1$ for
$k=r+1$ respectively, where $\mathcal{P}_{2n}$ is given in \eqref{eq11}.
\end{itemize}
\end{proposition}
\begin{proof}
We note that each  sequence $w=(\sigma,\nu)\in S_m^{2k}$ is periodic
and its period is  an even divisor of $2k$.  Let $w$ be the
$2n$-periodic sequence in $S_m^{2k}$ with $2n|2k$. Then   $w$ is
completely determined if the elements $(\sigma_i,\nu_i)$ is given
for $i=l+1,l+2,\cdots,l+2n$ for $l\in\mathbb{Z}$. It is obvious that
 $w$ belongs a $2n$-cycle of $\mathfrak{R}$. If $\mathcal{P}_{2n}$
denotes the number of $2n$-cycles of $\mathfrak{R}$ in $S_m^{2k}$,
then $\mathcal{P}_{2n}=\mathcal{N}_{2n}/(2n)$, where
$\mathcal{N}_{2n}$ is the number of $2n$-periodic sequences in
$S_m^{2k}$.

 (a) Suppose $(p,q,m)\in\Theta_1\cup\Theta_2$. If
 $k<r+1$, then it follows from Proposition \ref{p18}(b)(i) that
 there are $2^{n+1}$ ways of choosing the elements
 $(\sigma_{k-n+1},\nu_{k-n+1}),\\(\sigma_{k-n+2},\nu_{k-n+2}),\cdots,(\sigma_k,\nu_k),\cdots,(\sigma_{k+n},\nu_{k+n})$.
 Therefore
 \begin{equation*}
 \mathcal{N}_{2n}=2^{n+1}-\sum_{l|n,l\not=n}\mathcal{N}_{2l},
 \,\,\mathrm{and}\,\,\mathcal{P}_{2n}=\frac{1}{n}\left(2^n-\sum_{l|n,l\not=n}l\mathcal{P}_{2l}\right).
 \end{equation*}
 The statement (a) for $k<r+1$ follows by adding $\mathcal{P}_{2n}$
 for all divisors $2n$ of $2k$.

We note that the above computation are valid for the case $k=r+1$,
but, by Proposition \ref{p19}(b), we have to rule out the two
sequences  satisfying $\sigma_j=v_j$ for all $j\in \mathbb{Z}$.
Since these sequence belong to the same $2$-cycle of $\mathfrak{R}$,
the statement for $k=r+1$ follows.

(b). Suppose  $(p,q,m)\in\Theta_3\cup\Theta_4$. If $k<r+1$, then it follows from Proposition \ref{p18}(b)(ii) that
 there are $2^{n+1}$ ways of choosing the elements
 $(\sigma_{k-n+1},\nu_{k-n+1}),\\
 (\sigma_{k-n+2},\nu_{k-n+2})$,$\cdots,(\sigma_k,\nu_k),\cdots,(\sigma_{k+n},\nu_{k+n})$.
 Therefore
 \begin{equation*}
 \mathcal{N}_{2n}=2^{n+1}-\sum_{l|n,l\not=n}\mathcal{N}_{2l},
 \,\,\mathrm{and}\,\,\mathcal{P}_{2n}=\frac{1}{n}\left(2^n-\sum_{l|n,l\not=n}l\mathcal{P}_{2l}\right),
 \end{equation*}
 which proves (b) for $k<r+1$.

If $k=r+1$, then one gets $D_{pqm}^k$ by the same arguments as in (a).
\end{proof}

Next we are going to calculate  $E_{pqm}^k$. The following two
propositions have been obtained in \cite{LRR} for homogeneous
polynomial vector fields. They can be proved by the same arguments
as in \cite{LRR}.

\begin{proposition}\label{p25} Let $k\in J_{m,r}$. Then $C$ is
symmetrical if and only if there exists $w\in C$ such that
$\Psi(w)=\mathfrak{R}(w)$.
\end{proposition}

\begin{proposition}\label{p26}
Under the assumptions of Proposition \ref{p25}, if $C$ is a
symmetrical cycle, then there are exactly two elements of $C$ satisfying
$\Psi(w)=\mathfrak{R}(w)$.
\end{proposition}

\begin{proposition}\label{p27}
Let $0\not=k\in J_{m,r}$. Suppose that both $p$ and $q$ are odd.
\begin{itemize}
\item[(a)] If $r$ is odd and $k<r+1$, then
\begin{equation*}
E_{pqm}^k=\sum_{2n|k}I_{2n},\,\,\,\mathit{where}\,\,\,I_{2n}=2^{n+1}-\sum_{l|n,l\not=
n}I_{2l}.
\end{equation*}
\item[(b)] If $r$ is even and $k<r+1$, then
\begin{equation*}
E_{pqm}^k=\sum_{n|k}I_{2n},\,\,\,\mathit{where}\,\,\,I_{2n}=2^{(n+1)/2}-\sum_{l|n,l\not=
n}I_{2l}.
\end{equation*}
\item[(c)] If $r$ is odd,  then $E_{pqm}^{r+1}=(\sum_{2n|r+1}I_{2n})-1$,
where  $I_{2n}$ is defined  in (a).
\item[(d)] If $r$ is even,  then $E_{pqm}^{r+1}=(\sum_{n|r+1}I_{2n})-1$,
where $I_{2n}$ defined in (b).
\end{itemize}
\end{proposition}
\begin{proof}
It follows from Proposition \ref{p14} that $r$ is odd (resp. even)
if and only if $m$ is odd (resp. even). The proposition follows by
the same arguments as in \cite{LRR}.
\end{proof}

\begin{proposition}\label{p28}
Let $0\not=k\in J_{m,r}$. Suppose that  $p$ is  odd and $q$ is even.
\begin{itemize}
\item[(a)] If  $(p,q,m)\in\Theta_1\cup\Theta_2$, then
$E_{pqm}^k=2$ for $k<r+1$, and $E_{pqm}^{r+1}=1$ respectively.
\item[(b)] If
 $(p,q,m)\in\Theta_3\cup\Theta_4$, then
\begin{equation*}
E_{pqm}^k=\left\{\begin{array}{ll}4&\mathit{if}\,\,r\,\,\mathit{is\,\, odd},\,k<r+1,\\
2&\mathit{if}\,\,r\,\,\mathit{is\,\,
even},\,k<r+1.\end{array}\right.
\,\,\mathit{and}\,\,E_{pqm}^{r+1}=\left\{\begin{array}{ll}3&\mathit{if}\,\,r\,\,\mathit{is\,\, odd},\\
1&\mathit{if}\,\,r\,\,\mathit{is\,\, even},\end{array}\right.
\end{equation*}
 respectively.
\end{itemize}
\end{proposition}
\begin{proof}
We have known that the period of a symmetrical cycle of
$\mathfrak{R}$ is an even divisor of $2k$. From Proposition
\ref{p25} and Proposition \ref{p26}, the number  of symmetrical
$2n$-cycles of $\mathfrak{R}$ will be obtained if we divide by $2$
the number of $2n$-periodic sequence $w$ in $S_m^{2k}$ such that
$\Psi(w)=\mathfrak{R}(w)$. Since the sequences $w$ in a symmetric
cycle verify $\Psi(w)=\mathfrak{R}(w)$, we have
\begin{equation}\label{eq33}
w_{2k-i+1}=w_{i+1}.
\end{equation}

From now we suppose that $w$ belongs a symmetric cycle.

 (a)  If  $(p,q,m)\in\Theta_1\cup\Theta_2$ and $k<r+1$, then $w$
verifies Proposition \ref{p19}(ii), which gives $w_{2k-i+1}=-w_i$.
It follows from \eqref{eq33} that $w_{i+1}=-w_i$. This implies that
$w_i=(-1)^{i-1}w_1$. Hence $w$ is a $2$-periodic sequence. We can
take $4$ ways of choosing the element $w_1$. Therefore, $E_{pqm}^k=2$ for $k<r+1$.

Since the sequences such that   $\sigma_j=\nu_j$ for all $j\in
\mathbb{Z}$ verify $\Psi(w)=\mathfrak{R}(w)$ and belong to $2$-cycle
of $\mathfrak{R}$, we get $E_{pqm}^{r+1}=1$.

(b)  If $(p,q,m)\in\Theta_3\cup\Theta_4$, then $w$ verifies
Proposition \ref{p19}(iii), which gives $w_{2k-i+1}=w_{i-1}$ for
$i=2,3,\cdots,k$ and $w_{2k}=(-1)^{r-1}w_k$. It
follows from \eqref{eq33} that $w_{i+1}=w_{i-1}$ for
$i=2,3,\cdots, k$.

If $r$ is odd, then $w_k=w_{2k}$. It follows from Proposition
\ref{p15} that $k$ is even. The equation   $w_{i+1}=w_{i-1}$ for
$i=2,3,\cdots, k$  and \eqref{eq33} imply that
$w_{2k}=w_{2k-2}=\cdots=w_4=w_2$ and
$w_{2k-1}=w_{2k-3}=\cdots=w_3=w_1$.  Therefore, $w$ is a
$2$-periodic sequence. We can take $8$ ways of choosing the element
$w_1,\,w_2$ and hence $E_{pqm}^k=4$.

If $r$ is even, then it follows from Proposition \ref{p15} that $k$
is odd. The equation   $w_{i+1}=w_{i-1}$  and \eqref{eq33} imply that
$w_{2k}=w_{2k-2}=\cdots=w_4=w_2$ and
$w_{2k-1}=w_{2k-3}=\cdots=w_{k+2}=w_k=w_{k-2}=\cdots=w_3=w_1$. Since
$r$ is even, we obtain $w_{k}=-w_{2k}=-w_2$ from Proposition \ref{p19}(iii) and \eqref{eq33}, which gives $w_2=-w_1$. Therefore, $w$ is a
$2$-periodic sequence. We can take $4$ ways of choosing the element
$w_1$ and hence $E_{pqm}^k=2$.

We get   $E_{pqm}^{r+1}$ by the same arguments as in (a).

\end{proof}

In the end of this section we prove Theorem \ref{th6}.

\begin{proof}[Proof of Theorem \ref{th6}]  We use \eqref{cpqm} to get $C_{pqm}$.
 One obtains  $C_{pqm}^{k}$ for $k\not=0$  by   \eqref{eq32}, Proposition \ref{p23}, Proposition \ref{p24},
Proposition \ref{p27} and Proposition \ref{p28}. If $(p,q,m)\in\Theta_2\cup\Theta_3\cup\Theta_4$, then system \eqref{eq1} has at least
one singular point at infinity, which implies that $C_{pqm}^0=0$. If $(p,q,m)\in\Theta_1\backslash(\Theta_2\cup\Theta_3\cup\Theta_4)$, then
$C_{pqm}^0=2$ if $r$ is odd, and $C_{pqm}^0=0$ if $r$ is even, respectively. The statements of Theorem \ref{th6} follows from \eqref{cpqm}.
\end{proof}

\section{Local phase portrait at the origin and at infinity}\label{s5}

We shall prove Theorem \ref{th7} and Theorem \ref{th8} in this section.
\subsection{At the origin}

Consider the vector field $\bar X=(\bar P, \bar Q)$ with
\begin{equation}\label{local}
\bar P=\sum_{i\geq
m}P_i(x,y),\,\,\,\,\bar Q=\sum_{i\geq m}Q_i(x,y),
\end{equation}
where $P_i(x,y)$ and $Q_i(x,y)$ are $(p,q)$-
quasihomogeneous polynomials of degree $p-1+i$ and $q-1+i$  in the variables $x$ and $y$ respectively.

We would like to compare the local behaviour at origin of the vector field $\bar X$ given by \eqref{local} and the vector field $X_{m}=(P_m, Q_m)$.
For this we apply the quasihomogeneous blow-up method as doing in Section 2. In the  $(p,q)$-polar coordinates $(r,\phi)$ (see \eqref{eqpq})
the system associated to $\bar X$  is
\begin{equation*}
\dot r=\sum_{i \geq m} r^iF_i(\phi),\,\,\,\,\dot \phi=\sum_{i \geq m} r^{i-1}G_i(\phi),
\end{equation*}
where
\begin{equation}\label{FandG}
F_i(\phi)=\xi_i\left({\mathrm Cs}\phi,{\mathrm
Sn}\phi\right),\,\,G_i(\phi)=\eta_i\left({\mathrm Cs}\phi,{\mathrm
Sn}\phi\right),
\end{equation}
with
\begin{equation*} \begin{array}{ll}
 \xi_i(x,y)=& x^{2q-1}P_i(x,y)+y^{2p-1}Q_i(x,y)\\
 \eta_i(x,y)=& pxQ_i(x,y) - qyP_i(x,y)\end{array}
 \end{equation*}

Taking the changes \eqref{eqst}, the above system goes over to
\begin{equation}\label{eq37}
r'=\sum_{i \geq m}r^{i-m+1} f_i(\theta),\,\,\,\, \theta'=\sum_{i \geq m}r^{i-m}g_i(\theta).
\end{equation}
where prime denotes derivative with respect to $s$ and
\begin{equation}\label{eq38}
f_i(\theta)=F_i\left(\frac{{\mathcal
T}\theta}{2\pi}\right),\,\,\,g_i(\theta)=\frac{2\pi}{\mathcal
T}G_i\left(\frac{{\mathcal T}\theta}{2\pi}\right).
\end{equation}

Doing the same transformations in the vector field $X_m$ we obtain that it is equivalent to the following differential system
\begin{equation}\label{eq39}
\dot r= r f_m(\theta),\,\,\,\,\dot \phi= g_m(\theta),
\end{equation}
where $f_m(\theta)$ and $g_m(\theta)$ are defined in \eqref{eq38}.

\begin{proof}[Proof of Theorem \ref{th7}] We split the proof into two cases, as in Theorem \ref{th4}.

{\it Case 1. $X_m \in \Omega_{pqm}$ and $\eta_m(1,y)$ has no zeros, $\eta_m(0,1)\neq 0$.} Using  Theorem \ref{th3} and Theorem \ref{th4} we conclude that the origin  is a global focus (stable or unstable, depending on the sign of $I_{X_m}$) of $X_m$. It follows from \eqref{FandG} and \eqref{eq38} that $g_m(\theta)\not=0$ in this case.

As the critical point of $\bar X$ on $r=0$ are determined by the zeros of $g_m(\theta)$, $r=0$ is a periodic orbit for \eqref{eq37}. Furthermore, since the dominant terms of \eqref{eq37} in a neighborhood of $r=0$ are given by $r f_m$ and $g_m$, the orbit $r=0$ is a limit cycle with the same type of stability for \eqref{eq37} and \eqref{eq39}. Therefore the origin is a focus with the same type of stability for $\bar X$ and $X_m$ when $\eta_m(1,y)$ has no zeros and $\eta_m(0,1)\not=0$.

{\it Case 2. $X_m \in \Omega_{pqm}$, and at least one of the following condition is satisfied: (i) $\eta_m(1,y)$ has  zeros, (ii) $\eta_m(0,1)=0$.} From Theorem \ref{th4} all the zeros of $\eta(1,y)$ are simple if there exists, and $\partial \eta_m(0,1)/\partial x\not=0$ if $\eta_m(0,1)=0$, which implies that   all zeros of $g_m(\theta)$ are simple zeros for  $\theta\in[0,2\pi]$ by \eqref{FandG} and \eqref{eq38}.

Comparing \eqref{eq37} and \eqref{eq39} we observe that the critical point on $r=0$ are the same for both systems, and they are determined by the zeros of $g_m(\theta)$.  Furthermore the linear part of these systems in a critical point $(0,\theta^*)$ are
\begin{equation*}
\left(\begin{array}{cc}f_m(\theta^*)&0\\[2ex]
g_{m+1}(\theta^*)&g'_m(\theta^*)\end{array}\right)\,\,\,\,\,\,\mathrm{and}\,\,\,\,\,\,
\left(\begin{array}{cc}f_m(\theta^*)&0\\[2ex]
0&g'_m(\theta^*)\end{array}\right).
\end{equation*}
 By the definition of $f_m(\theta)$ and $g_m(\theta)$ we have that if $g_m(\theta^*)=0$, then  $f_m(\theta^*)\not=0$.
 So all the critical points of \eqref{eq37} and \eqref{eq39} are hyperbolic and the local phase portrait of \eqref{eq37} and \eqref{eq39} in $(0,\theta^*)$ are locally topologically equivalent, provided that $\theta^*$ is a simple zero of $g_m(\theta)$. As the same happens in each singular point and they determine the phase portrait of both systems in a neighborhood of $r=0$, we conclude that $\bar X$ and $X_m$ are locally topologically equivalent at the origin.
\end{proof}


\subsection{At infinity}

Now we can get the analogous of the Theorem \ref{th7} at infinity.

Consider the vector field $\hat X=(\hat P, \hat Q)$, with
\begin{equation}\label{eq40}
{\hat P} =\sum_{i=0}^
mP_i(x,y),\,\,\,\,\hat Q=\sum_{i=0}^mQ_i(x,y),
\end{equation}
where $P_i(x,y)$ and $Q_i(x,y)$ are  $(p,q)$-quasihomogeneous polynomials of degree $p-1+i$ and $q-1+i$ respectively.

Applying the  $(p,q)$-polar coordinates $(r,\phi)$ to the vector field $\hat X$ given by \eqref{eq40} the associated system is
\begin{equation*}
\displaystyle \dot r=\sum_{i=0}^m r^iF_i(\phi),\,\,\,\,\dot \phi=\sum_{i=0}^m r^{i-1}G_i(\phi),
\end{equation*}
where $F_i$ and $G_i$ are given in \eqref{FandG}.

Doing $\displaystyle \rho=\frac{1}{r}$ the system \eqref{eq40} is written as
\begin{equation}\label{eq41}
\displaystyle \dot \rho = - \sum_{i=0}^m \frac{1}{{\rho}^{i-2}} F_i(\phi),\,\,\,\,\dot \phi=\sum_{i=0}^m\frac{1}{{\rho}^{i-1}}G_i(\phi).
\end{equation}

Reparametrizing the time by $dt/ds={\rho^{m-1}}$ the system \eqref{eq41} becomes
\begin{equation}\label{eq42}
\displaystyle \begin{array}{l}
\rho' = - \rho ( F_m(\phi)+ \rho F_{m-1}(\phi) + \cdots + \rho^{m} F_0(\phi)),\\ \phi'= G_m(\phi)+ \rho G_{m-1}(\phi) + \cdots + \rho^{m} G_0(\phi).\end{array}
\end{equation}

\noindent where the derivative is with respect to $s$.

After these changes of coordinates the proof of Theorem \ref{th8} is analogous to the proof of Theorem \ref{th7}.

\begin{proof}[Proof of Theorem \ref{th8}]
Applying the same changes of coordinates described before to the vector field $X_m=(P_m, Q_m) \in \Omega_{pqm}$ we obtain
\begin{equation}\label{eq43}
\rho'=-\rho F_m(\phi), \,\,\,\,\, \phi'=G_m(\phi). \end{equation}

The infinity of $X_m$ and $\hat X$ corresponds to the invariant circle $\rho=0$ of the systems \eqref{eq43} and \eqref{eq42}.  As $X_m \in \Omega_{pqm}$, $\rho=0$ is a limit cycle or contains a finite number of hyperbolic critical points. Comparing the system \eqref{eq43} and \eqref{eq42} we observe that they have the same dominant terms. Therefore, it follows from the same argument used in the proof of Theorem \ref{th7}  that both systems are topologically equivalent in a neighborhood of $\rho=0$. So $X_m$ and $\hat X$ are topologically equivalent in a neighborhood of the infinity.
\end{proof}

The converse of Theorems \ref{th7} and Theorem \ref{th8} are not true. In \cite{LRR}, Proposition 19 is proved that the converse of an analogous theorems to homogeneous vector field are not true. As a homogeneous vector field is a $(1,1)$-quasi homogeneous vector field of degree $m$ the result follows.

\section*{Acknowledgements}
This work was done during a visit of the second  author to the Instituto de Ci\^encias Matem\'aticas e
de Computa\c c\~ao,  Universidade de S\~ao Paulo, in São Carlos, Brazil, and a visit of the first  author to Sun Yat-Sen University, Ghangzhou, China. The authors are grateful to these institutions for their support and hospitality.


\begin{thebibliography}{99}

\bibitem{andronov} {\sc A. A. Andronov and L. Pontrjagin}, {\em Systems grossiers}, Acad. Sci. URRS {\bf 14} (1937), 247--250.

\bibitem{AKL} {\sc J. C. Art\'es, R. Kooij, E. Robert and  J. Llibre },{\em Structurally stable quadratic vector fields}, Mem. Amer. Math. Soc. {\bf 134}(1998), viii+108pp.

\bibitem{BM} {\sc M. Brunella and M. Miari}, {\em Topological equivalence of a plane vestor field with its principal part defined throught Newton polyheda}, J.  Differential Equations {\bf 85} (1990), 338--366.

\bibitem{B} {\sc C. B. Collins}, {\em Algebraic classification of homogeneous polynomial vector fields in the plane}, Japan J. Indust. Appl. Math.
{\bf 49}(1997), 212--231.

\bibitem{CGP} {\sc  B. Coll, A. Gasull and P. Prohens}, {\em Differential equations defined by the sum of two quasihomogeneous vector fields}, Can. J. Math.
{\bf 13}(1996), 63--91.

\bibitem{CL} {\sc A. Cima and J. Llbre}, {\em Algebraic and topological classification of homogeneous cubic vector fields in the plane}, J. Math. Anal.
Appl. {\bf 147} (1990), 420--448.

\bibitem{chavaga} {\sc J. Chavarriga and I. A. Garcia}, {\em Lie symmetries of quasihomogeneous polynomial planar vector
fields and certain perturbations}, Acta Math. Sin. (Engl.Ser.) {\bf 21} (2005), 185-192.

\bibitem{dumortier-shafer} {\sc F. Dumortier and D. S. Shafer}, {\em Restriction on the equivalence homeomorphism in stability of polynomial vector fields} J. London Math. Soc. {\bf 41} (1990), 100--108.

\bibitem{DLA} {\sc F. Dumortier, J. Llibre, J. C. Art\'{e}s}, {\it Qualititive theory of planar differential systems}, Springer, 2006.

\bibitem{Fu} {\sc W. Fulton}, {\it Algebraic Curves. An Introduction to Algebraic Geometry}, Benjamin, New York, 1969.

\bibitem{KKN} {\sc J. Kotus, M. Krych and Z. Nitecki  }, {\em Global structural stability of flows on open surface}, ``Memoirs Amer. Math. Soc.",
Vol. 261,    Amer. Math. Soc. Providence, RI, 1982.

\bibitem{JL} {\sc X. Jarque and J. Llibre}, {\em Structural stability of planar Hamiltonian polynomial vector fields}, Proc. London Math. Soc.
{\bf 68} (1994), 617--640.

\bibitem{JLS1} {\sc X. Jarque, J. Llibre and D. S. Shafer}, {\em Structural stability of planar polynomial foliations}, J. Dynam. Differential Equations {\bf 17}(2005), 573--587.

\bibitem{JLS2} {\sc X. Jarque, J. Llibre and D. S. Shafer}, {\em Structurally stable quadratic foliations},  Rocky Mountain Journal of Mathematics {\bf 38}(2008), 489--530.

\bibitem{L} {\sc A.M. Liapunov}, {\em Stability of  motion}, in Mathematics in Science and Engineering, Vol. 30, Academic Press, 1966.

\bibitem{LLYZ} {\sc Weigu Li, J. Llibre, J. Yang and Z. Zhang}, {\em Limit cycles bifurcating from the period annulus of quasihomogeneous centers}, J. Dynam. Differential Equations {\bf 21}(2009), 133-152.

\bibitem{LRR} {\sc J. Llibre, Jesus S. Perez del Rio and J.A. Rodriguez}, {\em Structural stability of planar homogeneous
polynomial vector fields: applications to critical points and to infinity}, J. Differential Equations {\bf 125} (1996), 490--520.

\bibitem{LRR1} {\sc J. Llibre, Jesus S. Perez del Rio and J.A. Rodriguez}, {\em Structural stability of planar semi-homogeneous
polynomial vector fields: applications to critical points and to infinity}, Discrete and Continuous Dynamical Systems, {\bf 6} (2000), 809--828.

\bibitem{shafer} {\sc D. S. Shafer} {\em Structure and stability of gradient polynomial vector fields}, J. London Math. Soc. {\bf 41} (1990), 109--121.

\bibitem{peixoto} {\sc M.M. Peixoto}, {\em Structural stability on two-dimensional manifolds}, Topology {\bf 1} (1962), 101--120.


\end{thebibliography}
\end{document}